\newcommand{\R}{\mathds R}

\newcommand{\indM}{\mbox{$\mathfrak i_{\mathrm{Morse}}(M)$}}
\newcommand{\indS}{\mbox{$\mathfrak i_{\mathrm{Morse}}^\mathrm{s}(M)$}}
\newcommand{\indW}{\mbox{$\mathfrak i_{\mathrm{Morse}}^\mathrm{w}(M)$}}

\documentclass[twoside,draft,a4paper,10pt]{amsart}

\usepackage{times}
\usepackage{dsfont}
\usepackage[all]{xy}

\numberwithin{equation}{section}

\title[Bifurcation of CMC tori in spheres]{Bifurcation of constant mean curvature tori\\ in Euclidean spheres}
\author[L. J. Al\'\i as]{ Luis J. Al\'\i as}
\address{Departamento de Matem\'{a}ticas, \hfill\break\indent
Universidad de Murcia, Campus de Espinardo\hfill\break\indent 30100 Espinardo,
Murcia, \hfill\break\indent Spain} \email{ljalias@um.es}

\author[P.\ Piccione]{Paolo Piccione}
\address{Departamento de Matemática,\hfill\break\indent
Universidade de S\~ao Paulo, \hfill\break\indent Rua do Mat\~ao
1010,\hfill\break\indent CEP 05508-900, S\~ao Paulo, SP, Brazil}
\email{piccione.p@gmail.com}
\curraddr{Departamento de Matem\'{a}ticas, \hfill\break\indent
Universidad de Murcia, Campus de Espinardo\hfill\break\indent
30100 Espinardo, Murcia, \hfill\break\indent Spain}
\thanks{This work was partially supported by MEC project MTM2009-10418, and Fundaci\'{o}n S\'{e}neca
project 04540/GERM/06, Spain. This research is a result of the activity developed
within the framework of the Programme in Support of Excellence Groups of the
Regi\'{o}n de Murcia, Spain, by Fundaci\'{o}n S\'{e}neca, Regional Agency for
Science and Technology (Regional Plan for Science and Technology 2007-2010).}
\thanks{The second author is partially sponsored by Capes (Brazil), Grant BEX 1509-08-0, and Fundaci\'{o}n
S\'{e}neca grant 09708/IV2/08, Spain.}

\begin{document}


\theoremstyle{plain}\newtheorem*{teon}{Theorem}
\theoremstyle{definition}\newtheorem*{defin*}{Definition}
\theoremstyle{plain}\newtheorem{teo}{Theorem}[section]
\theoremstyle{plain}\newtheorem{prop}[teo]{Proposition}
\theoremstyle{plain}\newtheorem{lem}[teo]{Lemma}
\theoremstyle{plain}\newtheorem{cor}[teo]{Corollary}
\theoremstyle{definition}\newtheorem{defin}[teo]{Definition}
\theoremstyle{remark}\newtheorem{rem}[teo]{Remark}
\theoremstyle{plain} \newtheorem{assum}[teo]{Assumption}
\swapnumbers
\theoremstyle{definition}\newtheorem{example}{Example}[section]
\theoremstyle{plain} \newtheorem*{acknowledgement}{Acknowledgements}
\theoremstyle{definition}\newtheorem*{notation}{Notation}


\begin{abstract}
We use equivariant bifurcation theory to show the existence of infinite sequences
isometric embeddings of tori with constant mean curvature (CMC) in Euclidean spheres
that are not isometrically congruent to the CMC Clifford tori, and accumulating at
some CMC Clifford torus.
\end{abstract}

\maketitle
\tableofcontents

\begin{section}{Introduction}
\subsection{Index and stability of CMC hypersurfaces}
As is well known, minimal hypersurfaces in a Riemannian manifold $N$ are critical points
of the variational problem of minimizing area. Similarly, hypersurfaces with constant mean
curvature (CMC) in $N$ are also solutions to the same variational problem, when restricted
to volume-preserving variations. For such a critical point $M$, the stability equation of
the corresponding variational problem is given by the second variation of the area functional,
which can be written as the quadratic form $Q(f)=\int_M(Jf)\cdot f\;\mathrm{vol}_M$ acting on the space of functions
on $M$,
with $J=-\Delta-m\mathrm{Ric}_N(\vec{n})-\|S\|^2$.
Here $\Delta$ is the Laplacian on functions on $M$ relative to the induced metric,
$m=\mathrm{dim}(M)$, $\mathrm{Ric}_N(\vec{n})$ is the Ricci curvature of $N$ evaluated on
the unit normal field $\vec{n}$ and $S$ is the second fundamental form of the hypersurface. The operator $J$
is called the Jacobi or stability operator of the hypersurface.

In the case of minimal hypersurfaces, the index of a minimal hypersurface $M$, denoted by \indM, is
defined as the maximum dimension of any subspace on which $Q$ is negative definite. Equivalently, \indM\ is the number of negative eigenvalues of $J$
(counted with multiplicity), which is necessarily finite. Intuitively, \indM\ measures the number of independent directions
in which the hypersurface fails to minimize area. To see it, observe that if $Q(f)<0$ for a function $f$, then the second derivative
of the area functional in the normal variation of $M$ induced by $f$ is negative and therefore
$\mathrm{Area}(M)>\mathrm{Area}(M_t)$ for small values of $t$. That means that the minimal hypersurface $M$,
while a critical point of the area functional, is not a local minimum. For minimal hypersurfaces in the
Euclidean sphere $\mathbb{S}^{m+1}$ this is always the case, and $\indM\geq 1$ always.

A spectral analysis of the Jacobi differential operator of a given minimal/CMC submanifold
provides many information about the displacement of nearby minimal/CMC submanifolds.
In \cite{Si}, Simons characterized the totally geodesic equators $\mathbb{S}^{m}\subset\mathbb{S}^{m+1}$ as the
only compact minimal hypersurfaces in $\mathbb{S}^{m+1}$ having $\indM=1$. Later,
Urbano \cite{Ur}, when $m=2$, and El Soufi \cite{ElS}, for general $m$, proved that if $M$
is not a totally geodesic equator, then not only must be $\indM>1$ but in fact it must hold $\indM\geq m+3$.
On the other hand, apart from the totally geodesic equators, the easiest minimal
hypersurfaces in $\mathbb{S}^{m+1}$ are the minimal Clifford tori, and they all have $\indM=m+3$.
For that reason, it has been conjectured for a long time that minimal Clifford tori are the only compact
minimal hypersurfaces in $\mathbb{S}^{m+1}$ with $\indM=m+3$.
In \cite{Ur}, Urbano showed that the conjecture
is true when $m=2$. Later on, Guadalupe, Brasil Jr.\ and Delgado \cite{GBD} showed that the
conjecture is true for every dimension $m$, under the additional hypothesis of constant scalar
curvature of $M$. More recently, Perdomo \cite{Pe1} proved that the conjecture is also true for
every dimension $m$ with an additional assumption about the symmetries of $M$, and, in particular,
the conjecture is true for minimal hypersurfaces with antipodal symmetry.

In contrast to the case of minimal hypersurfaces, in the case of
hypersurfaces with constant mean curvature
one can consider two different eigenvalue problems: the usual Dirichlet problem, associated with the quadratic form $Q$
acting on the whole space of functions on $M$, and the so called \textit{twisted} Dirichlet problem, associated
with the same quadratic form $Q$, but restricted to the subspace of functions $f$
satisfying the additional condition $\int_Mf\mathrm{vol}_M=0$.
Similarly, there are two different notions of index, the \textit{strong index}, denoted by \indS\ and associated
to the usual Dirichlet problem, and the \textit{weak index}, denoted
by \indW\ and associated to the twisted Dirichlet problem.
Specifically, the strong index is simply the maximum dimension of any
subspace of functions on which $Q$ is negative definite. On the other hand, the weak index is the maximum dimension of
any subspace of functions satisfying $\int_Mf\mathrm{vol}_M=0$ on which $Q$ is negative definite. Obviously, from a
geometric point of view the weak index is more natural than the strong index.

Barbosa, do Carmo and Eschenburg \cite{BardoCEsc} characterized the totally umbilical spheres
$\mathbb{S}^{m}(r)\subset\mathbb{S}^{m+1}$ as the only compact CMC hypersurfaces in the Euclidean
sphere having $\indW=0$ (i.e., being weakly stable). In \cite{LuisAldirOscar}, Al\'\i as, Brasil and Perdomo proved that
the weak index of any other compact CMC hypersurface $M$ in $\mathbb{S}^{m+1}$ which is not
totally umbilical and has constant scalar curvature is greater than or equal to $m+2$,
with equality if and only if $M$ is a CMC Clifford torus
$\mathbb{S}^{j}(r)\times\mathbb{S}^{m-j}(\sqrt{1-r^2})$ with radius
$\sqrt{j/(m+2)}\leq r\leq\sqrt{(j+2)/(m+2)}$. More recently, in \cite{LuisAldirOscar2} the same authors complemented that estimate by showing that
the weak index of any compact CMC hypersurface $M$ in $\mathbb{S}^{m+1}$ which is neither
totally umbilical nor a CMC Clifford torus and has constant scalar curvature is greater than or equal to $2m+4$.
At this respect, it is worth pointing out that the weak stability index of the CMC Clifford torus
$\mathbb{S}^{j}(r)\times\mathbb{S}^{m-j}(\sqrt{1-r^2})$ depends on $r$ reaching its minimum value $m+2$ when
$r\in[\sqrt{j/(m+2)},\sqrt{(j+2)/(m+2)}]$, and converging to $+\infty$ as $r$ converges either to 0 or 1 (see
\cite{LuisAldirOscar} and Section \ref{sec:Clifford} for the details).

In this paper we study a result of existence for CMC embeddings of the torus $\mathbb S^j\times\mathbb S^{m-j}$
into the sphere $\mathbb S^{m+1}$, for arbitrary $1\le j<m$. There is a very rich literature on
CMC embeddings of the 2-torus $\mathbb S^1\times\mathbb S^1$ into $\mathbb S^3$. Reference
\cite{KilSch} contains an extensive description of many recent results of CMC embeddings of cylinders and tori
in the $3$-sphere, as well as a comprehensive list of references.
Two important conjectures are discussed in \cite{KilSch}: the Lawson conjecture and the Pinkall and Sterling conjecture.
The Lawson conjecture states that the only embedded minimal $2$-tori in $\mathbb S^3$ are isometrically
congruent to the minimal Clifford torus. The Pinkall and Sterling conjecture states that the only
embedded CMC tori in the $3$-sphere are rotational.

A complete classification of rotationally invariant CMC tori in $\mathbb S^3$ has been given recently
in \cite{HyndParkMcCuan}. The authors show the existence of a two-parameter family of CMC surfaces with
special spherical symmetry in $\mathbb S^3$, divided into five types (spheres, Clifford tori\footnote{Standard tori, in the language
of \cite{HyndParkMcCuan}}, catenoid-type,
unduloid-type and nodoid-type), which gives a remarkable
analogy with the classical Delaunay classification of rotationally symmetric CMC surfaces in $\R^3$.
\subsection{Bifurcation of CMC Clifford tori}
Let us fix notations in order to give a precise statement of the result proved in this paper.
Given a differentiable manifold $M$ and a Riemannian manifold $(N,g)$, we say that two immersions
$x_1,x_2:M\to N$ of $M$ into $N$ are \emph{isometrically congruent} when there is an isometry $\psi$
of $(N,g)$ that carries the image of $x_1$ onto the image of $x_2$, i.e., if there exists a
diffeomorphism $\phi$ of $M$ and an isometry $\psi$ of $(N,g)$ such that
the following diagram commutes:
\[\xymatrix{%
M\ar[r]^{x_1}\ar[d]_{\phi}&N\ar[d]^{\psi}\\
M\ar[r]_{x_2}&N}\]

Given integers $m>j\ge1$ and a positive real number $r\in\left]0,1\right[$, we will denote by
$x_r^{m,j}:\mathbb S^j\times\mathbb S^{m-j}\longrightarrow\mathbb S^{m+1}$
the embedding:
\begin{equation}\label{eq:defClifford}
\phantom{,\qquad p\in\mathbb S^j,\ q\in\mathbb S^{m-j}}
x_r^{m,j}(p,q)=\Big(r\cdot p,\sqrt{1-r^2}\cdot q\Big),\qquad p\in\mathbb S^j,\ q\in\mathbb S^{m-j}.\end{equation}
These are well known to be embeddings with constant mean curvature (CMC), called in the literature the
constant mean curvature \emph{Clifford tori}. This paper is devoted to proving the following:
\begin{teon}
For fixed $j$, $m$, there exist two sequences $\big(r_i^{m,j}\big)_{i\ge3}$ and $\big(s^{m,j}_l\big)_{l\ge3}$ in $\left]0,1\right[$,
specifically
\[
r_i^{m,j}=\sqrt{\frac{(i-2)(j+i-1)}{m-j+(i-2)(j+i-1)}}
\]
and
\[
s_l^{m,j}=\sqrt{\frac{j}{j+(l-2)(m-j+l-1)}},
\]
with $\lim\limits_{i\to\infty}r^{m,j}_i=1$ and $\lim\limits_{l\to\infty}s^{m,j}_l=0$,
such that the Clifford tori $x^{m,j}_{r_i}$ and $x^{m,j}_{s_l}$ are accumulation of pairwise non congruent CMC embeddings of $\mathbb S^j\times\mathbb S^{m-j}$ into $\mathbb S^{m+1}$, each of which is not congruent to any CMC Clifford
torus.\smallskip

For all other values of $r\in\left]0,1\right[$, the
family of CMC Clifford tori $x_r^{m,j}$ is \emph{locally rigid}, in the sense that any CMC embedding of
$\mathbb S^j\times\mathbb S^{m-j}$ into $\mathbb S^{m+1}$ which is \emph{sufficiently close} to $x_r^{m,j}$
must be isometrically congruent to an embedding of the CMC Clifford family.
\end{teon}
It must be observed that, in the case of the $3$-sphere $\mathbb S^3$ ($j=1$ and $m=2$ in the above statement),
our result can be obtained directly from the classification result in \cite{HyndParkMcCuan};
in this case, \cite{HyndParkMcCuan} provides an explicit description of the CMC tori in the bifurcating
branches, that are CMC embeddings of \emph{unduloid-type}.
\smallskip

The basic tool for proving our Theorem is variational bifurcation theory, which requires
a quite involved analytical and geometrical framework. A large part of the paper is devoted
to establishing the appropriate analytical and geometrical setup for applying bifurcation
and Morse theoretical techniques to the CMC variational problem. Let us describe in some
detail the technical issues that arise when one tries to give a formal proof of bifurcation
of CMC submanifolds.

The problem has a variational nature: CMC embeddings (in codimension one) are critical points
of the area functional defined in the space of embeddings that bound a region of fixed volume.
The set of trial maps for the variational problem should be a collection of embeddings of the torus
$M=\mathbb S^j\times\mathbb S^{m-j}$ into the sphere $N=\mathbb S^{m+1}$; in order to detect
solutions that are not isometrically congruent, one should take into consideration the action
of the diffeomorphism group of $M$, acting by right composition in the space of embeddings,
and the action of the isometry group of $N$, acting by left composition on the space of
embeddings. Note that the area and the volume functionals are invariant by the action of these
two groups. The action of the diffeomorphism group of $M$ on any set of embeddings of $M$
into $N$ is free, which suggests that one should consider a \emph{quotient} of the space
of embeddings by this action. This means that two embeddings $x,y:M\to N$ will be considered
equivalent if there exists a diffeomorphism $\phi:M\to M$ such that $y=x\circ\phi$.
As to the left action of the isometry group of $N$, this is not free; nevertheless, the group
is compact, and one can study a bifurcation problem for its critical orbits.
Thus, the variational problem has to be cast in a framework of \emph{equivariant bifurcation}
in a set of equivalence classes of embeddings of $M$ into $N$. One of the crucial issues
is the choice of which regularity has to be chosen in the set of embeddings.

In order to apply results from bifurcation theory, the set of trial maps for the variational
problem and the functionals involved must satisfy quite restrictive assumptions, more specifically,
a (local) Palais--Smale condition, and a Fredholm assumption on the second derivative. The latter is
needed for the computation of the local Morse theoretical invariants used to detect bifurcation:
roughly speaking, bifurcation occurs at an instant of a path of critical orbits when at this instant
there is a jump in the critical groups.

A purely Hilbert structure based on Sobolev spaces seems unfeasible for the type of variational
problem treated here. Namely, $H^1$-regularity is too loose when the dimension of the source
manifold $M$ is greater than one; weak $H^1$-solutions of the quasi-linear elliptic equation of
CMC embeddings may fail to be regular, and $H^1$-critical points of our variational problem
do not give in general CMC embeddings.
On the other hand, if one considers Sobolev regularity $H^k$, with $k>1$,
the second derivative of functionals of the type area/volume is never
Fredholm (but rather \emph{compact}).

The class of embeddings that seems most appropriate for developing a variational theory
of the area/volume functional is the set of embeddings of class $C^{2,\alpha}$, with $\alpha\in\left]0,1\right[$.
However, passing to such a Banach manifold setting has two annoying consequences. First, for the local
Palais--Smale condition, which is required in bifurcation theory, it is typically employed a Fredholm
assumption for the gradient of the function. Such assumption is satisfied when the second derivative
of the function is a Fredholm linear operator from the tangent space to its dual, but typically general Banach
spaces (and specifically $C^{2,\alpha}$) do not admit any Fredholm operator from the space to its dual.
This problem is solved here observing that the second variation of the area functional, which is given by a
second order linear elliptic operator, although it is not a Fredholm operator from $C^{2,\alpha}$ to its dual,
it is represented by a Fredholm operator from $C^{2,\alpha}$ to $C^{0,\alpha}$. The latter space
can be put in duality with $C^{2,\alpha}$ using the $L^2$-pairing, and in this way one obtains a sort
of gradient map which is in fact Fredholm. This yields a local Palais--Smale condition, as in
the case of Fredholm maps on Hilbert manifolds.

Second, the notion of nondegeneracy for critical points/orbits suitable for developing local Morse theory
in a Banach space context is delicate. Namely, one cannot require that the second variation be an isomorphism
from the tangent space to its dual, because again for general Banach spaces there exists no isomorphism
from the space to its dual. Recently, a computation of the critical groups for a certain variational
problem in Banach spaces has been performed under rather mild nondegeneracy assumtpions (simply,
injectivity of the Hessian, see \cite{CingVan1, CingVan2}), but the techniques in \cite{CingVan1, CingVan2}
do not extend in an obvious way to the problem discussed in the present paper.
Following \cite{Chang, Chang2, Uhl}, the appropriate notion of nondegeneracy for critical points of
functions in Banach manifolds has to be given in terms of a sort of splitting of the tangent space into
the sum of a strictly positive and a strictly negative subspaces for the Hessian.
We will use Chang's notion of \emph{s-nondegeneracy} for isolated critical points
of functionals on Banach manifolds (see \cite{Chang}), adapted to the case of isolated critical orbits.
This is proved here using an extension of the Hessian to the Sobolev space $H^1$, in which $C^{2,\alpha}$
sits continuously and densely, together with the fact that, by elliptic regularity, the eigenspaces
of such extension are contained in $C^{2,\alpha}$.

These questions are discussed in the first part of the paper, where we present an abstract framework
for bifurcation of critical orbits of constrained variational problems on Banach manifolds endowed
with a continuous action of a compact Lie group. The framework is described by a set of axioms; such
an abstract treatment is not intended with the purpose of developing a \emph{general theory}, but rather
to collect in a unified language the several results from different areas needed for our goal.

Other authors have studied bifurcation of CMC embeddings; for instance, in \cite{MazPac} (see also \cite{Ros}),
it is established a bifurcation result for \emph{nodoids}, that are immersed (not embedded)
CMC cylinders in $\R^3$. In the case of nodoids,  the hard part of the proof of bifurcation consists in the
analysis of the spectrum of the Jacobi operator. The variational setup in that case is quite straightforward,
because the ambient space is linear, and so the space of embeddings is an open subset of a Banach space.
In our case, the analysis of the spectrum of the Jacobi operator is easier, due to the fact that the
Laplacian and the second fundamental form of the CMC Clifford tori in the sphere are well understood
objects; for more general CMC tori in the sphere, the spectral analysis of the Jacobi operator
is more involved, see \cite{Ros, RosSul1, RosSul2}. On the other hand, the question
of differentiability for the space of embeddings of $M$ into $N$ modulo the diffeomorphisms of $M$ is
extremely involved. This issue is studied in detail in Ref.~\cite{AliasPiccione2010}.
It should also be observed that the result in \cite{MazPac} is obtained using
a classical bifurcation result of Crandall and Rabinowitz, see \cite{CranRab}, which allows
the authors to obtain also information on the smoothness of the bifurcating branch.
The bifurcation result of Crandall and Rabinowitz \emph{cannot} be employed in the variational problem
of the present paper, for the reason that Clifford tori never bifurcate from \emph{simple eigenvalues}, even working
equivariantly. The multiplicity of eigenvalues of CMC Clifford tori is computed in Proposition~\ref{thm:deg-Morseindex}.
More recently, a bifurcation result of Smoller and Wasserman \cite{SmoWas} has been employed
to determine bifurcation of fixed boundary nodoids in $\R^3$, see \cite{KoiPalPic}.
Smoller--Wasserman's deep result gives a sufficient condition for equivariant bifurcation that
applies also in situations where there is no jump of the Morse index, in terms of the group representation
on the negative eigenspace of the second variation. However, a basic assumption for applying such result
is that the original critical branch should consist of \emph{fixed points} for the group action.
This is not the case of the problem studied in this paper, in that CMC Clifford tori are not
invariant by the whole isometry group of Euclidean spheres: they have a large stabilizer
(see Proposition~\ref{thm:conststabilizer}), but their orbit is not a single point in the space
of unparameterized embeddings (Corollary~\ref{thm:dimensionorbit}).
It is also worth mentioning the paper \cite{jost} by Jost, Li-Jost and Peng, where
the authors study a (non equivariant) bifurcation problem for minimal embeddings and relate it to catastrophe theory.

The paper is divided into three parts. In Section~\ref{sec:abstractbifurcation} we will describe the
abstract variational setup that will be employed in the proof of our result. The material is this
section deals mostly with Morse theory in Banach manifolds and bifurcation theory, and it is obtained
by a systematization of several results already established in the literature in a form suited for
our purposes.

In Section~\ref{sec:varproblem} we will study the area and volume functionals for arbitrary embeddings
of a compact manifold $M$ into a Riemannian manifold $N$, and we will set the CMC embedding variational
problem into a Banach manifold framework.

Finally, in Section~\ref{sec:Clifford} we will study the specific case of embeddings of $\mathbb S^j\times\mathbb S^{m-j}$
into $\mathbb S^{m+1}$. We will show how the abstract bifurcation results apply to the case of CMC Clifford tori,
by studying degeneracy, Morse index and stabilizers of the embeddings \eqref{eq:defClifford},
finalizing the proof of our main result.

It should be observed that bifurcation theory does not provide a geometric description of the
CMC embeddings that bifurcate from the degenerate CMC Clifford tori.
\smallskip

\noindent\textbf{Acknowledgements.}\enspace
The authors gratefully acknowledge the help provided by Gerard Misio\l ek (University of Notre Dame, Notre Dame, IN, USA)
and Marco Degiovanni (Universit\`a Cattolica del Sacro Cuore, Brescia, Italy) during many
fruitful conversations. Gerard Misio\l ek has given several suggestions on the differentiable structure for
the set of smooth embeddings of a manifold $M$ into some other manifold $N$ modulo the group of diffeomorphisms
of $M$. Marco Degiovanni has helped the authors in the non smooth bifurcation results; among several suggestions,
he has pointed out reference \cite{DeVita} that contains the details of a proof of stability of critical groups
in the case of critical submanifolds for non smooth functionals on metric spaces.

\end{section}

\begin{section}{An abstract bifurcation setup for equivariant constrained variational problems}
\label{sec:abstractbifurcation}
\subsection{Equivariant constrained variational problems in a Banach setting}
Let us assume the following setup:
\begin{itemize}
\item[(A1)] $\mathfrak M$ is a differentiable manifold modeled on a Banach space $X$;
\item[(A2)] $G$ is  a compact connected Lie group  that acts continuously on $\mathfrak M$ by
homeomorphisms;
\item[(A3)] $\mathcal A:\mathfrak M\to\R$ is a smooth \emph{$G$-invariant} function, i.e.,
$\mathcal A(gx)=\mathcal A(x)$ for all $g\in G$ and all
$x$ in $\mathfrak M$.
\item[(A4)] $\mathcal V:\mathfrak M\to\R$ is a smooth $G$-invariant function without critical points, so that for all
$c\in\R$, the inverse image $\Sigma_c=\mathcal V^{-1}(c)$ is a smooth embedded $G$-invariant submanifold of
$\mathfrak M$.
\end{itemize}
For all $x\in\mathfrak M$, we will denote by $\mathcal O(x,G)$ the $G$-orbit of $x$, and
by $G_x$ the \emph{stabilizer} of $x$, which is the closed
subgroup of $G$ consisting of all the elements $g$ such that $gx=x$.
\smallskip

We are interested in studying constrained critical points of $\mathcal A$ with constraint $\mathcal V$, i.e.,
critical points of the restriction of $\mathcal A$ to the hypersurfaces $\Sigma_c=\mathcal V^{-1}(c)$, when $c$ varies
in $\R$.
Note that if $x_0$ belongs to $\Sigma_c$, then the orbit $\mathcal O(x_0,G)$ is entirely
contained in $\Sigma_c$. Moreover, if $x_0$ is a critical point of $\mathcal A\vert_{\Sigma_c}$,
$\mathcal O(x_0,G)$ consists entirely of critical points of $\mathcal A\vert_{\Sigma_c}$.
We say then that $\mathcal O(x_0,G)$ is a \emph{constrained critical orbit} of $\mathcal A$ subject to the
constraint $\mathcal V=\text{const}$. We will also set the following axiom:
\begin{itemize}
\item[(A5)] for every critical point $x_0$ of $\mathcal A$ subject to the
constraint $\mathcal V=\text{const}$, the orbit
$\mathcal O(x_0,G)$ is a smooth (finite dimensional compact) embedded submanifold of $\mathfrak M$, on
which the action of $G$ is smooth.
\end{itemize}
We observe that, by a result of Dancer \cite{Dan2}, assumption (A5) is always satisfied
when $\mathfrak M$ is a Banach space and the action of $G$ is by linear isomorphisms.

\subsection{Nondegeneracy and Morse index}\label{sub:nondegeneracy}
Consider a constrained critical orbit $\mathcal O(x_0,G)$ of $\mathcal A$ subject  to the
constraint $\mathcal V=\text{const.}$ with $\mathcal V(x_0)=c$.
The tangent space $T_{x_0}\mathcal O(x_0,G)$ is
contained in the kernel of the second variation $\mathrm d^2\big(\mathcal A\vert_{\Sigma_c}\big)$;
one possible notion of nondegeneracy for the orbit can be given by requiring
that the kernel of the bounded symmetric bilinear form $\mathrm d^2\big(\mathcal A\vert_{\Sigma_c}\big)(x_0)$
coincides with $T_{x_0}\mathcal O(x_0,G)$ in $T_{x_0}\Sigma_c$.

What we will be actually interested in is the question of accumulation of critical orbits
corresponding to different values of $c$, and this has to do with another type of degeneracy.
For $\lambda\in\R$, define $\mathcal A_\lambda:\mathfrak M\to\R$ by:
\begin{equation}\label{eq:defAlambda}
\mathcal A_\lambda=\mathcal A+\lambda\cdot\mathcal V.
\end{equation}
Recall that, by the Lagrange multiplier method, $x_0\in\Sigma_c$
is a critical point of the restriction $\mathcal A\vert_{\Sigma_c}$ if and only if there exists a real number $\lambda_0$
such that $x_0$ is a free critical point of the function
$\mathcal A_{\lambda_0}$ (note that $\lambda_0$ for all
the critical points of the orbit $\mathcal O(x_0,G)$).
Moreover, in this case the second variation of $\mathcal A\vert_{\Sigma_c}$ at $x_0$ is given by the
restriction of the second variation $\mathrm d^2\mathcal A_{\lambda_0}(x_0)$ to $T_{x_0}\Sigma_c$.
We will say that $\mathcal O(x_0,G)$ is a \emph{nondegenerate constrained critical orbit}
of $\mathcal A$ subject to the constraint $\mathcal V=\text{const.}$ if the kernel of the
symmetric bilinear form $\mathrm d^2\mathcal A_{\lambda_0}(x_0)$ in $T_{x_0}\mathfrak M$
is equal to $T_{x_0}\mathcal O(x_0,G)$. Note that this is a weaker notion of nondegeneracy compared
to the classical nondegeneracy of critical points for functions on a Hilbert manifold, where
one requires that the second derivative $\mathrm d^2\mathcal A_{\lambda_0}(x_0)$ be an isomorphism
from the tangent space $T_{x_0}\mathfrak M$ to its dual, and not just an injective map.
In the classical Hilbert setting, injectivity of the second derivative is not a condition
strong enough to develop Morse theory. Appropriate Fredholmness assumptions
will be imposed later, which will imply that our notion of nondegeneracy is suitable for
applying Morse theoretical techniques.

Similarly, there are two distinct notions of Morse index of a constrained critical orbits.
Given a (nondegenerate) constrained critical orbit $\mathcal O(x_0,G)$ with $\mathcal V(x_0)=c$
and with Lagrange multiplier $\lambda_0$,
we define the \emph{weak Morse index} $\mathfrak i_{\mathrm{Morse}}^\mathrm{w}\big(\mathcal O(x_0,G)\big)$
the (possibly infinite) dimension of a maximal subspace of $T_{x_0}\Sigma_c$ on which $\mathrm d^2\mathcal A_{\lambda_0}(x_0)$
is negative definite. By the \emph{strong Morse index}
$\mathfrak i_{\mathrm{Morse}}^\mathrm{s}\big(\mathcal O(x_0,G)\big)$
we mean the dimension of a maximal subspace of $T_{x_0}\mathfrak M$ on which $\mathrm d^2\mathcal A_{\lambda_0}(x_0)$
is negative definite. Since $\Sigma_c$ has codimension $1$ in $\mathfrak M$, it follows immediately
that either both the strong and the weak Morse index are infinite, or the following inequalities hold:
\begin{equation}\label{eq:wsindices}
\mathfrak i_{\mathrm{Morse}}^\mathrm{w}\big(\mathcal O(x_0,G)\big)\le \mathfrak i_{\mathrm{Morse}}^\mathrm{s}\big(\mathcal O(x_0,G)\big)
\le\mathfrak i_{\mathrm{Morse}}^\mathrm{w}\big(\mathcal O(x_0,G)\big)+1.\end{equation}
\subsection{Hilbertization and Fredholmness}
Although the basic framework for the variational problems considered in this paper
is given by manifolds modeled on Banach spaces, it will be important to have an
underlying Hilbert structure, on which the second derivative of our functionals
must be represented by self-adjoint Fredholm operators. In the purely Banach space context,
the assumption of Fredholmness for the second derivative of functionals is not reasonable,
as in most cases there exists no Fredholm operator between a Banach space and its dual.

A Hilbert/Fredholm structure will be employed in this paper in three different ways.
First, Fredholmness of the second derivative is used to prove a local compactness condition
(\emph{Palais--Smale condition}) for our functionals. Second,
for the proof of rigidity of the Clifford family away from the jumps of the Morse index
we will need to use a version of the implicit function theorem which requires Fredholmness.
Third, a Hilbert/Fredholm structure is needed in order to employ Morse
theoretical techniques in the Banach manifold context, in the spirit of \cite{Chang, Uhl}.

Two different sets of axioms are required. Axiom (HF-A) deals with the notion of \emph{gradient map}
for the functionals $\mathcal A_\lambda$, and it has the main purpose to guarantee a local
(PS) condition (see Proposition~\ref{thm:abstractPS}). Axiom (HF-B) deals with the second
derivative of $\mathcal A_\lambda$, and this is related to the computation of the local
homological invariants of Morse theory in the Banach manifold setting. It will
be used essentially in Proposition~\ref{thm:computCritGroups}.
The reader may keep in mind that, in our applications, the Banach space $X$ is given by the
space of $C^{k,\alpha}$ sections ($k\ge2$) of some Riemannian vector bundle $E$ over a compact manifold
$M$, the Banach space $Y$ is the space of $C^{k-2,\alpha}$ sections of $E$, $\mathcal H_0$ is the
space of $L^2$-sections of $E$ and $\mathcal H_1$ is the space of sections of $E$ having Sobolev
class $H^1$.
\smallskip
\begin{itemize}
\item[(HF-A)]
Given a critical point $x_0$ of $\mathcal A_{\lambda_0}$,
there exists an open neighborhood $U$ of $x_0$ diffeomorphic to an open subset of the Banach space $X$
(the model of $\mathfrak M$),
another Banach space $Y$ and a Hilbert space $\big(\mathcal H_0,\langle\cdot,\cdot\rangle_0\big)$, with continuous
inclusions
$X\hookrightarrow Y\hookrightarrow\mathcal H_0$ having dense images, and a map
\[H:\left]\lambda_0-\varepsilon,\lambda_0+\varepsilon\right[\times U\longrightarrow Y\] of class $C^1$ such that:
\[\mathrm d\mathcal A_{\lambda}(x)v=\big\langle H(\lambda,x),v\big\rangle_0,\]
for all $x\in U$, $\lambda\in\left]\lambda_0-\varepsilon,\lambda_0+\varepsilon\right[$ and $ v\in X$,
whose partial derivative:
\[\frac{\partial H}{\partial x}(\lambda_0,x_0):X\longrightarrow Y\]
is a Fredholm linear map of index $0$. We will call such  $H$ a \emph{gradient map} for the family $\mathcal A_\lambda$.
\medskip

\item[(HF-B)]
Given a critical point $x_0$ of $\mathcal A_{\lambda_0}$,
there exists a Hilbert space $\mathcal H_1$, a continuous inclusion $X\hookrightarrow\mathcal H_1$
having dense image, an  open neighborhood $U$ of $x_0$ in $\mathfrak M$ diffeomorphic to an open
subset $V$ of $X$, such that (identifying $U$ and $V$ with such diffeomorphism and considering
$\mathcal A_{\lambda_0}$ as a function on $V$):
\smallskip

\begin{itemize}
\item[(HF-B1)] for $x\in V$, the second derivative $\mathrm d^2\mathcal A_{\lambda_0}(x)$ admits an
extension to a bounded \emph{essentially positive}
symmetric bilinear form on $\mathcal H_1$, represented by the (self-adjoint) operator $S_{\lambda_0,x}$
on $\mathcal H_1$;
\item[(HF-B2)] there exists $\delta>0$ such that, for $\sigma\in\left]-\infty,\delta\right]$,
the $\sigma$-eigenspace of the essentially positive operator $S_{\lambda_0,x_0}$ is contained in $T_{x_0}\mathfrak M$.
\end{itemize}
\smallskip

\noindent We will also require that the above objects depend continuously
on $x_0$ and $\lambda_0$:
\smallskip

\begin{itemize}
\item[(HF-B3)] Axioms (HF-B1) and (HF-B2) hold for every $\lambda$ near $\lambda_0$ ,
and the map $S_{\lambda,x}$ depend continuously on $\lambda$  and $x$.
\end{itemize}
\end{itemize}
By an \emph{essentially positive} self-adjoint operator on a Hilbert space we mean an operator of the form
$P+K$, where $P$ is a positive isomorphism and $K$ is a self-adjoint compact operator. Equivalently,
essentially positive operators are self-adjoint Fredholm operators whose essential spectrum in contained
in $\left]0,+\infty\right[$. If $S$ is an essentially positive self-adjoint operator, then there exists
$\delta>0$ such that, if $\mathfrak s(S)\subset\R$ denotes the spectrum of $S$, the intersection
$\mathfrak s(S)\cap\left]-\infty,\delta\right]$ consists of a finite number of eigenvalues, each of which has finite multiplicity.
A symmetric bilinear form
$B$ on a Hilbert space is called essentially positive if $B=\langle S\cdot,\cdot\rangle$ for some essentially
positive operator $S$. This notion is independent on
the choice of a Hilbert space inner product $\langle\cdot,\cdot\rangle$
on $H$.

A few comments on the HF-axioms are in order.
First, we observe that the assumption on the density of the inclusions $X\hookrightarrow Y\hookrightarrow\mathcal H_0$
in (HF-A) implies that $\mathrm d\mathcal A_\lambda(x)=0$ if and only if the gradient $H(\lambda,x)$ vanishes.
Namely, if $\mathrm d\mathcal A_\lambda(x)=0$, then $H(\lambda,x)$ is orthogonal to the dense subspace $X$.

The second observation is that the strong nondegeneracy for a critical orbit $\mathcal O(x_0,G)$
is equivalent to the fact that the kernel of the derivative $\frac{\partial H}{\partial x}(\lambda_0,x_0)$
has dimension equal to the dimension of $\mathcal O(x_0,G)$. By the Fredholmness assumtpion, this implies
that, given any closed complement $X_2$ of $T_{x_0}\mathcal O(x_0,G)$ in $X$, the linear
map $\frac{\partial H}{\partial x}(\lambda_0,x_0)$ restricts to a Banach space isomorphism between
$X_2$ and the image of $\frac{\partial H}{\partial x}(\lambda_0,x_0)$.

Third, assumption (HF-B2) implies that $\mathcal O(x_0,G)$ is a nondegenerate critical orbit
for $\mathcal A_{\lambda_0}$ if the self-adjoint operator $S_{\lambda_0,x_0}$ is an isomorphism of $\mathcal H_1$.
Moreover, (HF-B1) and (HF-B2) imply that the strong Morse index $\mathfrak i_{\mathrm{Morse}}^\mathrm{s}\big(\mathcal O(x_0,G)\big)$
is finite and equal to the sum of the dimensions of the negative eigenspaces of $S_{\lambda_0,x_0}$.

Finally, it should be remarked that in specific examples the Hilbert spaces
$\mathcal H_0$ and $\mathcal H_1$ in Axioms (HF-A) and (HF-B) may be related.
For instance, in the situation described above where $X$ is the Banach space of $C^{2,\alpha}$-sections
of a Riemannian vector bundle $E$ over a compact manifold $M$, if the derivative $\frac{\partial H}{\partial x}(\lambda_0,x_0)$
is an elliptic operator, then the Hilbert space $\mathcal H_1$ can be defined as the Hilbert space completion
of $X$ with respect to the inner product \[\langle x_1,x_2\rangle_1=
\langle x_1,x_2\rangle_0+\left\langle\tfrac{\partial H}{\partial x}(\lambda_0,x_0)x_1,
x_2\right\rangle_0.\]

\subsection{Pseudo-critical points}
Let us use the axiom (HF-A) to show a preliminary result on the distribution of critical points
of the family $\mathcal A_\lambda$. In the variational setup (A1)---(A5), assume that (HF-A) holds
in a neighborhood $U$ critical point $x_0$ of $\mathcal A_{\lambda_0}$. Set:
\[X_1=\mathrm{Ker}\left[\frac{\partial H}{\partial x}(\lambda_0,x_0)\right],
\qquad Y_2=\mathrm{Im}\left[\frac{\partial H}{\partial x}(\lambda_0,x_0)\right];\]
let $X_2$ be a closed complement of $X_1$ in $X$ and $Y_1$ be a closed complement of $Y_2$ in $Y$.
Note that $\mathrm{dim}(X_1)=\mathrm{dim}(Y_1)=d$, by the zero index assumption on the Fredholm map $\frac{\partial H}{\partial x}(\lambda_0,x_0)$.
Let $P_2:Y\to Y_2$ be the projection relative to the direct sum decomposition $Y=Y_1\oplus Y_2$, and
set $H_2=P_2\circ H$. A point $(\bar\lambda,\bar x)$ in the domain of $H$ will be called a \emph{pseudo-critical
point} for the family $\mathcal A_\lambda$ if $H_2(\bar\lambda,\bar x)=0$. Under certain assumptions (see Proposition~\ref{thm:bifimpliesdegen}),
we will show that pseudo-critical points are in fact critical. Next Lemma tells us how pseudo-critical
points are displaced near a nondegenerate critical orbit.
\begin{lem}\label{thm:pseudocritical}
In the variational setup (A1)---(A5), let $x_0\in\mathfrak M$ be a critical point of $\mathcal A_{\lambda_0}$
whose critical orbit $\mathcal O(x_0,G)$ has dimension $d$ and is nondegenerate. Assume that
(HF-A) holds at $x_0$.
Then, there exists an open neighborhood $W$ of $(\lambda_0,x_0)$ in $\R\times\mathfrak M$ and
a $(d+1)$-dimensional submanifold $\mathcal D\subset W$ such that, for $(\bar\lambda,\bar x)\in  W$,
$(\bar\lambda,\bar x)$ is a pseudo-critical point for the family $\mathcal A_\lambda$
if and only if $(\bar\lambda,\bar x)\in\mathcal D$.
\end{lem}
\begin{proof}
Let $a_0\in X_1$ and $b_0\in X_2$ be such that $x_0=a_0+b_0$.
Consider the $C^1$-map $\mathcal F$, defined in a neighborhood of $(\lambda_0,a_0,b_0)$ in
$\R\times X_1\times X_2$ and taking values in a neighborhood of $(\lambda_0,a_0,0)$ in $\R\times X_1\times Y_2$,
obtained by setting:
\[\mathcal F(\lambda,a,b)=\big(\lambda,a,H_2(\lambda,a+b)\big).\]
We claim that $\mathcal F$ is a diffeomorphism around the point $(\lambda_0,a_0,b_0)$; in order to
prove the claim it suffices to apply the Inverse Mapping Theorem, observing that the differential
$\mathrm d\mathcal F(\lambda_0,a_0,0)$ is written in block form as:
\[\mathrm d\mathcal F(\lambda_0,a_0,0)=\begin{pmatrix}1&0&0\cr0&\mathrm{Id}&0\cr *&*&T\end{pmatrix},\]
where $T=\frac{\partial H_2}{\partial b}(\lambda_0,a_0,b_0):X_2\to Y_2$, and $\mathrm{Id}$ is the identity map of $X_1$.
Now, $\mathrm d\mathcal F(\lambda_0,a_0,0)$ is an isomorphism, as $T$ is an isomorphism. Notice in fact that
$T$ is the restriction of $\frac{\partial H}{\partial x}(\lambda_0,x_0)$ to $X_1$.
 Then, the set of pseudo-critical points
for the family $\mathcal A_\lambda$ in a small neighborhood of $(\lambda_0,a_0,b_0)$ is the graph
of the $C^1$-function $\psi$, defined in a neighborhood of $(\lambda_0,a_0)$ in $\R\times X_1$ and taking
values in a neighborhood of $b_0$ in $X_2$, defined by $\psi(\lambda,a)=\pi\big(\mathcal F^{-1}(\lambda,a,0)\big)$,
where $\pi:\R\times X_1\times X_2\to X_2$ is the projection onto the third coordinate.
The graph of $\psi$ is a $(d+1)$-dimensional $C^1$-submanifold of $\R\times X$, which is identified
with a $(d+1)$-dimensional $C^1$-submanifold $\mathcal D$ of $\R\times\mathfrak M$.
\end{proof}
\subsection{The local Palais--Smale condition}
Recall that a \emph{Palais--Smale sequence} for the functional $\mathcal A_\lambda$ is a sequence
$(x_n)_{n\in\mathds N}$ such that $\vert\mathcal A_\lambda(x_n)\vert$ is bounded and
$\Vert\mathrm d\mathcal A(x_n)\Vert$ is infinitesimal. Given a (closed) subset $\mathfrak C\subset\mathfrak M$,
the functional $\mathcal A_\lambda$ is said to satisfy the \emph{Palais--Smale condition} in $\mathfrak C$ if
every Palais--Smale sequence for $\mathcal A_\lambda$ contained in $\mathfrak C$ has a converging subsequence.

An adaptation of a classical result of Marino and Prodi (see \cite{MarProd}), gives
the following:
\begin{prop}\label{thm:abstractPS}
In the variational setup described by (A1)---(A5), assume that (HF-A) is satisfied
at every point of a constrained critical orbit $\mathcal O(x_0,G)$ of the functional $\mathcal A_{\lambda_0}$.
Then, given $\varepsilon>0$ sufficiently small, there exists a closed neighborhood
$W$ of $\mathcal O(x_0,G)$ such that, for all $\lambda\in\left]\lambda_0-\varepsilon,\lambda_0+\varepsilon\right[$,
$\mathcal A_\lambda$ satisfies the Palais--Smale condition in $W$.
\end{prop}
\begin{proof}
Using the compactness of the orbit and the $G$-equivariance, it suffices to show the existence
of a closed neighborhood of $x_0$ on which $\mathcal A_\lambda$ satisfies the Palais--Smale condition
for all $\lambda$ sufficiently close to $\lambda_0$.
Since the set of Fredholm operators is open in the space of all bounded operators from $X$ to $Y$
and the map $H$ is of class $C^1$, by taking $\varepsilon>0$ and $U$ sufficiently small
we can assume that the partial derivative $\frac{\partial H}{\partial x}(\lambda,z)$ is Fredholm
for all $z\in U$. The local form of a $C^1$-map between Banach spaces with Fredholm derivative (see for
instance \cite[Theorem~1.7, p.\ 4]{AbrRob}) says that there is a $C^1$-change of coordinates that
carries a neighborhood of $x_0$ in $U$ to a neighborhood of zero of a Banach space direct sum
$E_1^\lambda\oplus E_2^\lambda$, with $\mathrm{dim}(E_2^\lambda)<+\infty$, and that takes $x_0$ to $(0,0)$,
and a $C^1$-change of coordinates that carries a neighborhood of  $H(\lambda_0,x_0)$ in $Y$
to a neighborhood of zero of another Banach space direct sum $F_1^\lambda\oplus E_1^\lambda$ with
$\mathrm{dim}(F_1^\lambda)<+\infty$,
and that takes $H(\lambda_0,x_0)$ to $(0,0)$, such that, using these coordinates,
the map $H(\lambda,\cdot)$ takes the form $E_1^\lambda\oplus E_2^\lambda\ni(u,v)\mapsto
\big(\eta_\lambda(u,v),u\big)\in F_1^\lambda\oplus E_1^\lambda$,
where $\eta_\lambda:E_1^\lambda\oplus E_2^\lambda\to F_1^\lambda$ a $C^1$-map with $\mathrm d\eta_\lambda(0,0)=0$.
It is immediate so see that such a map is \emph{proper}
when restricted to the unit ball of $E_1^\lambda\oplus E_2^\lambda$, thus the map
$H(\lambda,\cdot)$ is proper when restricted to a suitable
closed neighborhood $W^\lambda$ of $x_0$, depending on $\lambda$.
The size of $W^\lambda$ depends continuously on $H$ and $\frac{\partial H}{\partial x}$;
a proof of this assertion is obtained easily by keeping track of sizes in the proof of
\cite[Theorem~1.7, p.\ 4]{AbrRob}, that uses the inverse function theorem. Thus, using the continuity of
$H$ and $\frac{\partial H}{\partial x}$, one can find a fixed neighborhood $W_0$ of $x_0$ such that
for all $\lambda$ sufficiently close to $\lambda_0$, $H(\lambda,\cdot)$ is proper when restricted to $W_0$.

Given a Palais--Smale sequence $(x_n)_{n\in\mathds N}$ for $\mathcal A_\lambda$
contained in such neighborhood, then $H(\lambda,x_n)$ tends to $0$ as $n\to\infty$, and thus
the set $K=\big\{H(\lambda,x_n):n\in\mathds N\big\}\bigcup\{0\big\}$ is compact. The sequence
$(x_n)_{n\in\mathds N}$ is contained in the compact subset $H(\lambda,\cdot)^{-1}(K)$, and therefore
it admits a converging subsequence, which concludes the proof.
\end{proof}
\subsection{Local Morse invariants}\label{sub:localMorsetheory}
Let $x_0$ be a critical point of $\mathcal A_{\lambda_0}$, and let
$\mathcal O=\mathcal O(x_0,G)$ be the critical orbit of $x_0$;
set $c=\mathcal A_{\lambda_0}(x_0)$.
For $q\in\R$, define $\mathcal A_{\lambda_0}^q$ the closed sublevel:
\[\mathcal A_{\lambda_0}^q=\big\{x\in\mathfrak M:\mathcal A_{\lambda_0}(x)\le q\big\}.\]
Given a coefficient ring $\mathds F$, the sequence of \emph{$\mathds F$-critical groups} of $\mathcal O$
is the sequence \[\mathfrak H_*(\mathcal O,\lambda_0;\mathds F)=
\big(\mathfrak H_\nu(\mathcal O,\lambda_0;\mathds F)\big)_{\nu\in\mathds N}\]
of relative (singular) homology groups with coefficients in $\mathds F$:
\[\mathfrak H_\nu(\mathcal O,\lambda_0;\mathds F)=H_\nu\big(\mathcal A_{\lambda_0}^c,\mathcal A_{\lambda_0}^c\setminus\mathcal O;\mathds F).\]
By excision, if $U$ is any open set of $\mathfrak M$ that contains $\mathcal O$, then:
\[\mathfrak H_\nu(\mathcal O,\lambda_0;\mathds F)=
H_\nu\big(\mathcal A_{\lambda_0}^c\cap U,(\mathcal A_{\lambda_0}^c\cap U)\setminus\mathcal O;\mathds F)\]
for all $\nu\in\mathds N$.

We will now show how are assumptions can be used to compute the critical groups of a nondegenerate
orbit. Such computation is based on an analysis of appropriate $G$-invariant neighborhoods of
a critical orbit; it will be useful to employ some terminology from principal fiber bundles.
Recall that given an $H$-principal bundle
$P\to X$ over the manifold $X$, and given a topological space $Y$ endowed with an $H$-action, the \emph{twisted product}
$P\times_HY$ is a fiber bundle over $X$ whose fiber at $x\in X$ in the quotient of the product
$P_x\times Y$ by the left action of $H$ given by:
\[H\times(P_x\times Y)\ni\big(h,(p,y)\big)\mapsto(ph^{-1},hy)\in P_x\times Y.\]
Since the right action of $H$ on $P_x$ is free and transitive, such quotient is homeomorphic
to $Y$; $P\times_YX$ is fiber bundle over $X$ with typical fiber $Y$.
When $G$ is a compact Lie group acting on a completely regular topological space $X$,
then through every $x\in X$ there is a \emph{slice} $\Sigma_x$ (see \cite[Chapter~II, Sec.~5]{Bredon});
recall that $\Sigma_x$ is a subset of $X$, which is invariant by the action of the stabilizer
$H$ of $x$, and such that the map $G\times\Sigma_x\ni(g,y)\mapsto g\cdot y\in X$ defines
an homeomorphism of the twisted product $G\times_H\Sigma_x$ and an open neighborhood
of the orbit $Gx$.

For the computation of the critical groups in our setup, we will consider
for simplicity the field $\mathds F=\mathds Z_2$.
\begin{prop}\label{thm:computCritGroups}
In the variational setup (A1)---(A5), let $x_0\in\mathfrak M$ be a critical point
of $\mathcal A_{\lambda_0}$. Assume that (HF-B) holds around the points of the  critical
orbit $\mathcal O=\mathcal O(x_0,G)$, that (HF-A) holds at every point of $\mathcal O$
and that $\mathcal O(x_0,G)$ is nondegenerate.
Set $\mu=\mathfrak i_{\mathrm{Morse}}^\mathrm{s}\big(\mathcal O(x_0,G)\big)$. Then for
all $\nu\in\mathds N$, the critical group $\mathfrak H_\nu(\mathcal O;\mathds Z_2)$
is isomorphic to the singular homology group $H_{\nu-\mu}(\mathcal O;\mathds Z_2)$
of the orbit $\mathcal O$.
\end{prop}
\begin{proof}
Let $H$ be the stabilizer of $x_0$, and consider the principal fiber bundle $G\longrightarrow G/H\cong\mathcal O(x_0,G)$.
Let $\Sigma_{x_0}$ be a slice at $x_0$, let $U$ be the open neighborhood of $\mathcal O$
given by $G\cdot\Sigma_{x_0}$ and set $c=\mathcal A_{\lambda_0}(x_0)$.
The intersection $\mathcal A_{\lambda_0}^c\cap U$ is a fiber bundle over $\mathcal O$, whose
typical fiber is the intersection $\Sigma_{x_0}^c=\mathcal A_{\lambda_0}^c\cap\Sigma_{x_0}$.
Let $\mathcal S$ be a submanifold of $\mathfrak M$ through $x_0$ which is transversal to
$\mathcal O$ at $x_0$ (i.e., $T_{x_0}\mathcal S\cap T_{x_0}\mathcal O=\{0\}$ and
$T_{x_0}\mathcal S+T_{x_0}\mathcal O=T_{x_0}\mathfrak M$); consider a homeomorphism
$\mathfrak h$ from $\Sigma_{x_0}$ to $\mathcal S$, with $\mathfrak h(x_0)=x_0$,
and such that $\mathcal A_{\lambda_0}\circ\mathfrak h=\mathcal A_{\lambda_0}$ on $\Sigma_{x_0}$.
Then, $\mathfrak h$ carries $\mathcal A_{\lambda_0}^c\cap\Sigma_{x_0}$ to
 $\mathcal A_{\lambda_0}^c\cap\mathcal S$ and $(\mathcal A_{\lambda_0}^c\cap\Sigma_{x_0})\setminus\{x_0\}$ to
 $\mathcal A_{\lambda_0}^c\cap\mathcal S\setminus\{x_0\}$, thus:
 \[H_k\big((\mathcal A_{\lambda_0}^c\cap\Sigma_{x_0}),(\mathcal A_{\lambda_0}^c\cap\Sigma_{x_0})\setminus\{x_0\};\mathds Z_2\big)
 \cong
 H_k\big((\mathcal A_{\lambda_0}^c\cap\mathcal S),(\mathcal A_{\lambda_0}^c\cap\mathcal S)\setminus\{x_0\};\mathds Z_2\big)\]
 for all $k$.
The restriction of $\mathcal A_{\lambda_0}$ to $\mathcal S$ is a smooth function, and it has
$x_0$ as an isolated critical point. This follows easily from the fact that, by the $G$-invariance, the
critical points of the restriction of $\mathcal A_{\lambda_0}$ to $\mathcal S$ are precisely
the critical points of $\mathcal A_{\lambda_0}$ that lie on $\mathcal S$. Then, since
$\mathcal O(x_0,G)$ is nondegenerate, it is an isolated critical orbit of $\mathcal A_{\lambda_0}$,
which has an isolated intersection with $\mathcal S$ at $x_0$, by transversality.

Now, the critical point $x_0$ of the restriction of $\mathcal A_{\lambda_0}$ to $\mathcal S$
is \emph{s-nondegenerate}, in the sense of \cite{Chang}. Recall that this means that there exists
a diffeomorphism from an open subset $\mathcal V$
of the Banach space $T_{x_0}\mathcal S$ to an open neighborhood of $x_0$ in $\mathcal S$ and a hyperbolic isomorphism
$L:T_{x_0}\mathcal S\to T_{x_0}\mathcal S$ such that, using such diffeomorphism to identify
$\mathcal A_{\lambda_0}$ with a smooth function on $\mathcal V$, the following conditions are satisfied:
\begin{itemize}
\item[(a)] $\mathrm d^2\mathcal A_{\lambda_0}(x_0)\big[Lv,w\big]=\mathrm d^2\mathcal A_{\lambda_0}(x_0)\big[v,Lw\big]$
for all $v,w\in T_{x_0}\mathcal S$;
\item[(b)] $\mathrm d^2\mathcal A_{\lambda_0}(x_0)\big[Lv,v\big]>0$ for all $v\in T_{x_0}\mathcal S$, $v\ne0$;
\item[(c)] $\mathrm d\mathcal A_{\lambda_0}(x)\big[L(x-x_0)\big]>0$ for all $x$ in $\mathcal V\setminus\{x_0\}$
with $\mathcal A_{\lambda_0}(x)\le\mathcal A_{\lambda_0}(x_0)$.
\end{itemize}
We will use Axioms (HF-B) to verify s-nondegeneracy. Consider the open subset $V$ of the Banach space $X\cong T_{x_0}\mathfrak M$,
the Hilbert space $\mathcal H_1$ and the essentially positive operator $S_{x_0,\lambda_0}$ on $\mathcal H_1$
as in (HF-B).
Denote by $T_{x_0}\mathcal O^\perp$ the orthogonal complement of $T_{x_0}\mathcal O$ in
$\mathcal H_1$. Then, setting $\overline X=X\cap T_{x_0}\mathcal O^\perp$, one has
$X=T_{x_0}\mathcal O\oplus\overline X$, because $T_{x_0}\mathcal O\subset X$;
moreover $\overline X$ is isomorphic to $T_{x_0}\mathcal S$.
Choose a local chart of $\mathfrak M$ around $x_0$ taking values in $T_{x_0}\mathcal O\oplus\overline X$
and carrying an open neighborhood of $x_0$ in $\mathcal S$ to an open neighborhood of $0$ in
$\{0\}\oplus\overline X\cong\overline X\cong T_{x_0}\mathcal S$.

The Hilbert space $T_{x_0}\mathcal O^\perp$ is invariant by $S_{\lambda_0,x_0}$,
and it splits as an orthogonal direct sum $\mathcal H_-\oplus\mathcal H_*\oplus\mathcal H_+$,
where $\mathcal H_-$ is finite dimensional and it is spanned by the eigenvectors of
$S_{\lambda_0,x_0}$ having negative eigenvalue, $\mathcal H_*$ is finite dimensional
and it is spanned by the eigenvectors of
$S_{\lambda_0,x_0}$ having eigenvalue in $\left]0,\delta\right]$, and
$\mathcal H_+$ has infinite dimension, is invariant by $S_{\lambda_0,x_0}$ and
the restriction of
$S_{\lambda_0,x_0}$ to $\mathcal H_+$ has spectrum contained in $\left]\delta,+\infty\right[$.
By assumption (HB-3), $\mathcal H_-$ and $\mathcal H_*$ are contained in $\overline X$, and
thus $\overline X=\mathcal H_-\oplus\mathcal H_*\oplus(\overline X\cap\mathcal H_+)$.
Let $L:\overline X\to\overline X$ be the isomorphism whose restriction to $\mathcal H_-$ is
minus the identity, and whose restriction to $\mathcal H_*\oplus(\overline X\cap\mathcal H_+)$
is the identity. Evidently, $L$ is hyperbolic. Such operator admits an extension to a self-adjoint
isomorphism of $\mathcal H_1$, and properties (a) and (b) above are readily verified
for such extension. Namely, it is easy to see that $L$ commutes with $S_{\lambda_0,x_0}$,
which implies that (a) holds. Property (b) is obvious, using the fact that the spaces
$\mathcal H_-$ and $\mathcal H_*\oplus(\overline X\cap\mathcal H_+)$ are orthogonal and invariant
by $S_{\lambda_0,x_0}$; the composition $S_{\lambda_0,x_0}L$ is a positive isomorphism of $\mathcal H_1$.
Property (c) is also obtained easily using
the mean value theorem for the function $t\mapsto\mathrm d\mathcal A_{\lambda_0}(x_0+tv)\big[Lv\big]$,
where $v=x-x_0$ and $x$ is near $x_0$. Namely, $\mathrm d\mathcal A_{\lambda_0}(x)\big[Lv\big]=
\mathrm d\mathcal A_{\lambda_0}(x)\big[Lv\big]-\mathrm d\mathcal A_{\lambda_0}(x_0)\big[Lv\big]=
\mathrm d^2\mathcal A_{\lambda_0}(x_0+\bar tv)\big[v,Lv]=\left\langle S_{\lambda_0,x_0+\bar tv}Lv,v\right\rangle_1$ for some $\bar t\in[0,1]$.
Since $S_{\lambda_0,x_0}L$ is a positive isomorphism of $\mathcal H_1$, then by continuity
$S_{\lambda_0,x_0+\bar tv}L$ is a positive isomorphism for $x$ near $x_0$.
Thus, $\left\langle S_{\lambda_0,x_0+\bar tv}Lv,v\right\rangle_1>0$ for $v\ne0$, and (c) holds.

By Proposition~\ref{thm:abstractPS}, $\mathcal A_{\lambda_0}$ satisfies
the Palais--Smale condition in a closed neighborhood of $x_0$ in $\mathfrak M$. Using the transversality of
$\mathcal S$ to the critical orbit $\mathcal O$, it follows that also the restriction of $\mathcal A_{\lambda_0}$
to a closed neighborhood of $x_0$ in $\mathcal S$ satisfies the Palais--Smale condition.
We can therefore apply Chang's result on Morse theory in Banach manifolds applied
to the restriction of $\mathcal A_{\lambda_0}$ to $\mathcal S$; by \cite[Theorem~1]{Chang},
we have:
\begin{equation}\label{eq:homoloyfiber}
H_k\big(\mathcal A_{\lambda_0}^c\cap\mathcal S,(\mathcal A_{\lambda_0}^c\cap\mathcal S)\setminus\{x_0\};\mathds Z_2\big)\cong
\begin{cases}\mathds Z_2,&\text{if $k=\mu$;}\cr0,&\text{if $k\ne\mu$.}\end{cases}
\end{equation}

With this, in order to compute the critical groups of the orbit $\mathcal O$ we can use
an abstract result on the homology of fiber bundles. One has a bundle pair $(E,\dot E)$ on the manifold
$\mathcal O$, where $E=\mathcal A_{\lambda_0}^c\cap U$ has typical fiber $F=\mathcal A_{\lambda_0}^c\cap\Sigma_{x_0}$
and $\dot E=\big(\mathcal A_{\lambda_0}^c\cap U\big)\setminus\mathcal O$ has typical fiber
$\dot F=\big(\mathcal A_{\lambda_0}^c\cap\Sigma_{x_0}\big)\setminus\{x_0\}$.
Using the \emph{Leray--Hirsch theorem} (see \cite[Theorem~9]{Spanier}), we have the following isomorphism:
\begin{equation}\label{eq:LerayHirsch}
H_k(E,\dot E;\mathds Z_2)\cong\bigoplus_{i=0}^nH_i(F,\dot F;\mathds Z_2)
\otimes_{\mathds Z_2}H_{n-i}(\mathcal O;\mathds Z_2).
\end{equation}
Leray--Hirsch theorem uses two assumptions that are easily verified in our case. First, the relative homology
$H_i(F,\dot F;\mathds Z_2)$ has to be finite dimensional for all $i$. Second, a technical condition called
\emph{cohomology extension of the fiber} has to be satisfied. When $E$ is (homotopic to)
an open neighborhood of the zero section
of a vector bundle and $\dot E$ is $E$ minus the zero section, then the cohomology extension of the fiber exists
always when the coefficient field is $\mathds Z_2$. From \eqref{eq:homoloyfiber} and \eqref{eq:LerayHirsch}
we compute easily:
\begin{multline*}
\mathfrak H_k(\mathcal O;\mathds Z_2)=
H_k\big(\mathcal A_{\lambda_0}^c\cap U,(\mathcal A_{\lambda_0}^c\cap U)\setminus\mathcal O;\mathds Z_2\big)
\\
\cong\bigoplus_{i=0}^kH_i\big(\mathcal A_{\lambda_0}^c\cap\Sigma_{x_0},(\mathcal A_{\lambda_0}^c\cap\Sigma_{x_0})\setminus\{x_0\};\mathds Z_2\big)
\otimes_{\mathds Z_2}H_{k-i}(\mathcal O;\mathds Z_2)
\\
\cong\bigoplus_{i=0}^kH_i\big(\mathcal A_{\lambda_0}^c\cap\mathcal S,(\mathcal A_{\lambda_0}^c\cap\mathcal S)\setminus\{x_0\};\mathds Z_2\big)
\otimes_{\mathds Z_2}H_{k-i}(\mathcal O;\mathds Z_2)\\
=H_\mu\big(\mathcal A_{\lambda_0}^c\cap\mathcal S,(\mathcal A_{\lambda_0}^c\cap\mathcal S)\setminus\{x_0\};\mathds Z_2\big)
\otimes_{\mathds Z_2}H_{k-\mu}(\mathcal O;\mathds Z_2)\cong H_{k-\mu}(\mathcal O;\mathds Z_2).
\end{multline*}
This concludes the proof.
\end{proof}
\subsection{Equivariant constrained bifurcation}
In the variational setup (A1)---(A5),
let us now assume that we have a path of constrained critical orbits, as follows:
\begin{itemize}
\item[(B1)] $[a,b]\ni r\mapsto\lambda_r\in\R$ is a map of class $C^1$, with derivative $\lambda'_r>0$ for all $r$;
\smallskip

\item[(B2)] $[a,b]\ni r\mapsto x_r\in\mathfrak M$ is a map of class $C^1$, and $x_r$ is a critical point of
$\mathcal A_r=\mathcal A_{\lambda_r}$ for all $r$.
\end{itemize}
The assumption $\lambda_r'>0$ in (B1) implies that the image of the map $\lambda$ is an interval $[c,d]$, and
that there exists the inverse function $[c,d]\ni\lambda\mapsto r_\lambda\in[a,b]$.

We say that an instant $\bar r\in[a,b]$ is an \emph{constrained critical orbit bifurcation instant}
if there exists a sequence $r_n$ tending to $\bar r$,  and a sequence $x_n\in\mathfrak M$
tending to $x_{\bar r}$ as $n$ tends to infinity, such that:
\begin{itemize}
\item $\mathrm d\mathcal A_{\lambda_{r_n}}(x_{n})=0$;
\item $x_{n}\not\in\mathcal O(x_{r_n},G)$,
\end{itemize}
for all $n\in\mathds N$. In other words, $\bar r$ is a constrained critical orbit bifurcation
instant if arbitrarily close to $\mathcal O(x_{\bar r},G)$ one can find other constrained critical
orbits that do not belong to the given path of constrained critical orbits.
\smallskip

Under suitable assumptions, degeneracy is a necessary condition for bifurcation.
This is obtained by refining the result of Lemma~\ref{thm:pseudocritical}, as follows:
\begin{prop}\label{thm:bifimpliesdegen}
Consider the variational setup (A1)---(A5), and (B1), (B2).
Let $\bar r\in\left]a,b\right[$ be fixed; assume that $\mathcal O(x_{\bar r},G)$ is
a nondegenerate critical orbit of $\mathcal A_{\lambda_{\bar r}}$ and that (HF-A) holds around
every point of $\mathcal O(x_{\bar r},G)$. If:
\begin{itemize}
\item[(C1)] the (connected component of the identity of the)
isotropy group $G_{x_r}$ is constant for $r$ near $\bar r$;
\item[(C2)] the curve $r\mapsto x_r$ is not tangent to $\mathcal O(x_{\bar r},G)$ at $x_{\bar r}$.
\end{itemize}
Then:
\begin{itemize}
\item[(a)] for $\varepsilon>0$ small enough, the set \[\mathcal O_\varepsilon=\bigcup
\limits_{r\in \left]\bar r-\varepsilon,\bar r+\varepsilon\right[}\{\lambda_r\}\times
\mathcal O(x_r,G)\] is a $C^1$-submanifold of $\R\times\mathfrak M$ of dimension $d+1$, where $d=\mathrm{dim}(G/G_{x_{\bar r}})$;
\smallskip

\item[(b)] for $(\lambda,x)$ near $\big(\lambda_{\bar r},\mathcal O(x_{\bar r},G)\big)$, $x$ is a critical point
of $\mathcal A_{\lambda}$ if and only if $(\lambda,x)\in\mathcal O_\varepsilon$.
\end{itemize}
In particular, constrained critical orbit bifurcation does not occur at $\bar r$.
\end{prop}
\begin{proof}
By equivariance, it suffices to study the problem in the neighborhood $W$ of $\big(\lambda_{\bar r},x_{\bar r})$
where the thesis of Lemma~\ref{thm:pseudocritical} holds.
The set $\mathcal O_\varepsilon\cap W$ is contained in the $(d+1)$-dimensional $C^1$-submanifold
$\mathcal D$ consisting of pseudo-critical points, whose existence in proven in Lemma~\ref{thm:pseudocritical}.
Moreover, $\mathcal O_\varepsilon\cap W$ is foliated by the sets $\mathcal O(x_{r_\lambda},G)\cap W$,
for $\lambda\in\left]\lambda_{\bar r-\varepsilon},\lambda_{\bar r+\varepsilon}\right[$, that are
submanifolds of dimension $d$ of $\R\times\mathfrak M$ by (A5) and (C1), and it contains the curve
$\lambda\mapsto x_{r_\lambda}$ which not tangent
to the orbit $\mathcal O(x_{\bar r},G)$ by (C2).
The foliation $\lambda\mapsto\mathcal O(x_{r_\lambda},G)$ is continuous, by (C1) and the fact
the action of $G$ on $\mathfrak M$ is continuous.
It follows that $\mathcal O_\varepsilon\cap W=\mathcal D\cap\big(\left]\lambda_{\bar r-\varepsilon},\lambda_{\bar r
+\varepsilon}\right[\times\mathfrak M\big)$ for $\varepsilon>0$ small enough, and the conclusion follows immediately from Lemma~\ref{thm:pseudocritical}.
\end{proof}
\begin{rem}
When $\mathfrak M$ is a Banach space and the action of $G$ is by linear isomorphisms of $\mathfrak M$,
the result can be proven using a $G$-equivariant version of the infinite dimensional Implicit Function
Theorem, as proved for instance in \cite{Dan1, Dan2, ReckPet}, replacing assumption (C1) with an
the algebraic assumption that the \emph{isotropic representation} of $G_{x_{\bar r}}$ on
$T_{x_{\bar r}}\mathcal O(x_{\bar r},G)$ should not have non zero fixed point.
This is also equivalent to the fact that the stabilizer $G_{x_{\bar r}}$ has the same dimension of its normalizer
in $G$.
\end{rem}
\begin{prop}\label{thm:MorseIndexConstant}
Consider the variational setup (A1)---(A5), and (B1), (B2).
Let $I\subset[a,b]$ be an interval such that, for all $r\in I$, $\mathcal O(x_r,G)$ is nondegenerate,
the critical orbit $\mathcal O(x_r,G)$ has constant dimension, and such
that (HF-B) holds around every $x_r$. Then, the strong Morse index $\mathfrak i_{\mathrm{Morse}}^\mathrm{s}\big(\mathcal O(x_r,G)\big)$
is constant on $I$.
\end{prop}
\begin{proof}
The integer valued functions \emph{index} and \emph{index plus nullity} are respectively lower and upper
semi-continuous in the space of essentially positive Fredholm operators on a Hilbert space
(see for instance \cite[Corollary~2.8]{JavMasPic}).
Our assumptions imply that the nullity of the second variation of the functionals $\mathcal A_{\lambda_r}$
 is constant on $I$. Thus, the index function $\mathfrak i_{\mathrm{Morse}}^\mathrm{s}\big(\mathcal O(x_r,G)\big)$
 is both lower and upper semi-continuous on $I$, hence constant.
 \end{proof}
Degeneracy is only a necessary condition for orbit bifurcation. In order to guarantee bifurcation
a sufficient condition is that the degeneracy occurs at an instant $\bar r$ where the critical groups
of the orbit $\mathcal O(x_{\bar r},G)$ have a discontinuity. Under assumptions (C1) and (C2), such
discontinuity occurs at the instants where the strong Morse index has a jump.
There are several results available for equivariant bifurcation, but they don't quite fit
into our framework;  for instance, in \cite{SmoWas} it is studied equivariant bifurcation from
a branch of isolated critical points, i.e., from critical orbits consisting of just one point.
In view to our application, in which  some global regularity assumptions will be dropped, it will
be convenient to give a proof using a non smooth approach.
\begin{teo}\label{thm:maincriticalorbitbifurcation}
Consider the variational setup (A1)---(A5), and (B1), (B2); for $r\in[a,b]$ set:
\[\mu_r=\mathfrak i_{\mathrm{Morse}}^\mathrm{s}\big(\mathcal O(x_r,G)\big).\]

Let $\bar r\in\left]a,b\right[$ be a given instant, and assume the following:
\begin{itemize}
\item axiom (HF-A) holds at all points of $\mathcal O(x_{\bar r},G)$;
\item axiom (HF-B) holds in a neighborhood of $\mathcal O(x_{\bar r},G)$;
\item assumptions (C1) and (C2) hold;
\item[(D1)] for $r\ne\bar r$ the critical orbit $\mathcal O(x_r,G)$ is nondegenerate;
\item[(D2)]
\label{itm:assjumpmorseindex} for $\varepsilon>0$ small, $\mu_{\bar r-\varepsilon}\ne\mu_{\bar r+\varepsilon}$.
\end{itemize}
Then, $\bar r$ is a constrained critical orbit bifurcation instant.
\end{teo}
\begin{proof}
Note that, by Proposition~\ref{thm:MorseIndexConstant},
$\mu_r$ is constant on the left and on the right of $\bar r$, and thus assumption
(D2) does not depend on the choice of $\varepsilon$ small enough.
By Proposition~\ref{thm:abstractPS}, for $r$ near $\bar r$ the functional $\mathcal A_{\lambda_r}$ satisfies
the Palais--Smale condition on a closed neighborhood of the orbit $\mathcal O(x_{\bar r},G)$.
Moreover, the assumption (D2) on the jump of the Morse index and the fact that, by (C1),
the critical orbits $\mathcal O_r=\mathcal O(x_{r},G)$, have the same dimension $d$ (in fact,
they are all diffeomorphic if the whole isotropy group $G_{x_r}$ is constant near $\bar r$),
we have that the sequence of critical groups
$\mathfrak H_*\big(\mathcal O_r,\mathds Z_2\big)$ has a jump at $r=\bar r$. Namely, set $\mu_\pm=\mu_{\bar r\pm\varepsilon}$,
and assume $\mu_+>\mu_-$.
Then, using Proposition~\ref{thm:computCritGroups}, for $\varepsilon>0$ small enough we have:
\begin{equation}\label{eq:contradict1}
\mathfrak H_{d+\mu_+}\big(\mathcal O_{\bar r+\varepsilon},\mathds Z_2\big)\cong H_d(\mathcal O_{\bar r+\varepsilon},\mathds Z_2)\cong\mathds Z_2,
\end{equation}
while
\begin{equation}\label{eq:contradict2}
\mathfrak H_{d+\mu_+}\big(\mathcal O_{\bar r-\varepsilon},\mathds Z_2\big)
\cong H_{d+\mu_+-\mu_-}(\mathcal O_{\bar r-\varepsilon},\mathds Z_2)=\{0\}.
\end{equation}
Similarly, if $\mu_+<\mu_-$, then $\mathfrak H_{d+\mu_-}\big(\mathcal O_{\bar r+\varepsilon},\mathds Z_2\big)=0$
while $\mathfrak H_{d+\mu_-}\big(\mathcal O_{\bar r-\varepsilon},\mathds Z_2\big)\cong\mathds Z_2$.

Set $\bar\lambda=\lambda_{\bar r}$ and let $\varepsilon'$ be such that
$[\bar\lambda-\varepsilon',\bar\lambda+\varepsilon']\subset[\lambda_{\bar r-\varepsilon},\lambda_{\bar r+\varepsilon}]$.
An equivariant version\footnote{A complete proof of a stability result for the critical groups
in the case of critical submanifolds, that can be adapted to our situation, is found in the unpublished work \cite{DeVita}.} of the stability result for critical groups given in
\cite[Theorem~5.2]{CorvHant} says that, when the local Palais--Smale condition holds in some closed
neighborhood $W$ of $\mathcal O_{\bar r}$, if the following two assumptions are satisfied:
\begin{itemize}
\item the only critical orbit of $\mathcal A_\lambda$ in $W$ is $\mathcal O_{r_\lambda}$
for $\lambda\in[\bar\lambda-\varepsilon',\bar\lambda+\varepsilon']$;
\item the critical orbits $\mathcal O_{{\bar r-\varepsilon}}$ and $\mathcal O_{{\bar r+\varepsilon}}$
are nondegenerate,
\end{itemize}
then
\begin{equation}\label{eq:stabilitycritgroups}
\mathfrak H_*\big(\mathcal O_{{\bar r-\varepsilon}},\mathds Z_2\big)=
\mathfrak H_*\big(\mathcal O_{{\bar r+\varepsilon}},\mathds Z_2\big).
\end{equation}
Clearly, \eqref{eq:stabilitycritgroups} is in contradiction with \eqref{eq:contradict1}
and \eqref{eq:contradict2}, which shows that, for every sufficiently small
closed neighborhood of $\mathcal O_{\bar r}$ and every sufficiently small $\varepsilon'>0$,
some $\mathcal A_\lambda$, with $\lambda\in[\bar\lambda-\varepsilon', \bar\lambda+\varepsilon']$, admits a critical orbit contained in $W$ and
distinct from $\mathcal O_{r_\lambda}$. This means that critical orbit bifurcation occurs at $r=\bar r$.
\end{proof}
Observe that, by \eqref{eq:wsindices}, assumption~(D2) in Theorem~\ref{thm:maincriticalorbitbifurcation}
holds if the weak Morse index has a jump of at least two at $\bar r$, i.e., if:
\[\Big\vert\mathfrak i_{\mathrm{Morse}}^\mathrm{w}\big(\mathcal O(x_{\bar r-\varepsilon},G)\big)
- \mathfrak i_{\mathrm{Morse}}^\mathrm{w}\big(\mathcal O(x_{\bar r+\varepsilon},G)\big)\Big\vert\ge2.\]

Theorem~\ref{thm:maincriticalorbitbifurcation} will in fact be used in the present paper
under the slightly weaker assumption that the manifold $\mathfrak M$ has only a local differentiable
structure defined by an atlas of charts that are continuously compatible, and that the
functions $\mathcal A$ and $\mathcal V$ are smooth in these local charts.
The purely topological nature of the global aspects of the proof of Theorem~\ref{thm:maincriticalorbitbifurcation},
namely, the computation of the local Morse invariants (Proposition~\ref{thm:computCritGroups})
and the stability result for critical groups, makes it clear that the result of
Theorem~\ref{thm:maincriticalorbitbifurcation} holds under these more general assumptions.

\end{section}
\begin{section}{On the variational problem of CMC hypersurfaces}
\label{sec:varproblem}
Let us formalize the question of obtaining constant mean curvature hypersurfaces of a Riemannian
manifold as critical points of the area functional restricted to \emph{volume preserving} variations,
in the spirit of \cite{BardoCEsc}.
Let $(N,g)$ be a connected Riemannian manifold, and let $M$ be a connected
compact differentiable manifold
with $\mathrm{dim}(N)=\mathrm{dim}(M)+1$. We will assume for simplicity that $M$ and $N$ are oriented, although
the entire theory can be developed also in the non orientable case. Let $\mathrm{vol}_g$ denote the
volume form on $N$ associated to the metric $g$.
\subsection{The manifold of unparameterized embeddings}
\label{sub:unparameterizedembeddings}
Let $k\ge2$ and $\alpha\in\left]0,1\right[$ be fixed. Consider the smooth Banach manifold
$C^{k,\alpha}(M,N)$ of all maps $x:M\to N$ of H\"older class $C^{k,\alpha}$, and
let $\mathrm{Emb}(M,N)$ be the open subset of $C^{k,\alpha}(M,N)$ consisting of embeddings.
For the main results of the present paper, we will consider the case $k=2$.
The set $\widetilde{\mathrm{Emb}}(M,N)$ is the set of equivalence classes $\big\{[x]:x\in\mathrm{Emb}(M,N)\big\}$,
where two embeddings $x,y\in\mathrm{Emb}(M,N)$ are equivalent if there exists a diffeomorphism
$\phi:M\to M$ such that $x=y\circ\phi$. Elements of $\widetilde{\mathrm{Emb}}(M,N)$ are called
\emph{unparameterized embeddings} of class $C^{k,\alpha}$ of $M$ into $N$.
The topology of $\widetilde{\mathrm{Emb}}(M,N)$ is the quotient topology induced from ${\mathrm{Emb}}(M,N)$.
The geometrical structure of $\widetilde{\mathrm{Emb}}(M,N)$ is studied in \cite{AliasPiccione2010}.
Let us recall here that $\widetilde{\mathrm{Emb}}(M,N)$ has a natural atlas of charts that makes into
an infinite dimensional topological Banach manifold. The charts of this atlas are of the following form.
Given a \emph{smooth} (i.e., $C^\infty$) embedding $x:M\to N$, there exists an open neighborhood
$\widetilde{\mathcal U}_x$ of $[x]$ in $\widetilde{\mathrm{Emb}}(M,N)$ and an open neighborhood ${\mathcal W}_x$
of the zero section of the Banach space of $C^{k,\alpha}$ section of the normal bundle $x^\perp$ of $x$,
and a bijection $\widetilde\Phi_x:\widetilde{\mathcal U}_x\to\widetilde{\mathcal W}_x$ defined by
$\widetilde\Phi_x\big([y]\big)=u$, where the map $z:M\to N$ given by $z(p)=\exp_{x(p)}\big(u(p)\big)$, $p\in M$,
is an embedding of class $C^{k,\alpha}$ which is equivalent to $y$.

As $x$ runs in the set of smooth embeddings, the maps $\widetilde\Phi_x$ form an atlas of charts for $\widetilde{\mathrm{Emb}}(M,N)$
that are continuously compatible. If $F:\mathrm{Emb}(M,N)\to\R$ is a smooth function which is invariant
by reparamaterization, i.e., $F(x\circ\phi)=F(x)$ for all diffeomorphism $\phi:M\to M$, then
the induced map $\widetilde F:\widetilde{\mathrm{Emb}}(M,N)\to\R$ is \emph{smooth in every local chart},
i.e., $\widetilde F\circ\widetilde\Phi_x^{-1}$ is smooth for all $x$.
Thus, one has a well defined notion of \emph{critical point} for a smooth function $\widetilde F$
on $\widetilde{\mathrm{Emb}}(M,N)$, as well as a natural notion of second derivative
$\mathrm d^2\widetilde F([x_0])$ at a critical point $[x_0]$.

Moreover, the smooth action of the isometry group of $(N,g)$ on $\mathrm{Emb}(M,N)$ by right composition
induces a continuous action on $\widetilde{\mathrm{Emb}}(M,N)$. Given a smooth $x$, the $[x]$-orbit of
this action is a smooth submanifold of $\widetilde{\mathrm{Emb}}(M,N)$ in local charts.

Details are found in \cite{AliasPiccione2010}.

\subsection{The area functional}
Given an embedding $x:M\to N$, one can define the \emph{area} $\mathcal A(x)$ of $x$ as the volume
of $M$ relatively to the volume form $x^*(\mathrm{vol}_g)$, which is the pull-back of $\mathrm{vol}_g$ by $x$:
\[\mathcal A(x)=\int_Mx^*(\mathrm{vol}_g).\]
If an auxiliary Riemannian metric $h$ is fixed on $M$, then $\mathcal A(x)$ can be written more explicitly as:
\begin{equation}\label{eq:areafunctional}
\mathcal A(x)=\int_M\left(\mathrm{det}\big(\mathrm dx(p)^*\mathrm dx(p)\big)\right)^{1/2}\,\mathrm{vol}_h,
\end{equation}
where $\mathrm dx(p)^*$ is the adjoint of the linear map $\mathrm dx(p):T_pM\to T_{x(p)}N$, adjoint taken
relatively to the scalar products $h_p$ on $T_pM$ and $g_{x(p)}$ on $T_{x(p)}N$.
Note that $\mathcal A$ can be seen as a functional on the manifold $\mathrm{Emb}(M,N)$, and
it is invariant by the group of diffeomorphisms of $M$; thus, $\mathcal A$ gives a well
defined functional on the quotient space $\widetilde{\mathrm{Emb}}(M,N)$,
still denoted by $\mathcal A$ with a slight abuse of notations.
Using the local charts of $\widetilde{\mathrm{Emb}}(M,N)$ described above,
$\mathcal A$ is a smooth function in the neighborhood of every smooth embedding.
More precisely:
\begin{prop}\label{thm:critpointsA}
Let $x$ be a $C^\infty$ embedding of $M$ into $N$; let $\big(\widetilde{\mathcal U}_x,\widetilde\Phi_x\big)$
be the local chart of $\widetilde{\mathrm{Emb}}(M,N)$ around $[x]$ described in Subsection~\ref{sub:unparameterizedembeddings}.
Then, the map $\mathcal A_x=\mathcal A\circ\widetilde\Phi_x^{-1}:\widetilde\Phi_x
\big(\widetilde{\mathcal U}_x\big)\to\R^+$
is smooth, and $u$ is a critical point of this functional if and only if $\widetilde\Phi_x^{-1}(u)$
is the class $[y]\in\widetilde{\mathcal U}_x\subset\widetilde{\mathrm{Emb}}(M,N)$
of a \emph{minimal} embedding $y:M\to N$.
\end{prop}
\begin{proof}
First, note that the functional $\mathcal A$ is smooth on $\mathrm{Emb}(M,N)$. Namely,
in the local charts described in Subsection~\ref{sub:unparameterizedembeddings}, $\mathcal A$ is given by the composition
of a nonlinear first order differential operator defined on $C^{k,\alpha}$-sections of $x^*(TN)$
with the linear map \[C^0(M,\R)\ni f\longmapsto\int_Mf\mathrm{vol}_h\in\R.\]
Second, observe that $\mathcal A$ is invariant by diffeomorphisms of $M$, i.e., the area of an embedding does not depend on its parametrization.
Smoothness of $\mathcal A_x$ follows now easily using the results of \cite{AliasPiccione2010}.
It is well known (see for instance \cite{Lawson}) that $0$ is a critical point of $\mathcal A_x$
if and only if $x$ is a minimal embedding.
\end{proof}
\subsection{Volume of a region with boundary $\mathbf{x(M)}$.}
Let us now look at the variational problem of constant mean curvature embeddings.
Before we go into the general case, let us first discuss an instructive problem of
establishing when a given embedding $x:M\to N$ has image which is the boundary of
an open subset of $N$. Equivalently, denoting by $M_0\subset N$ the image $x(M)$,
we want to know when the open set $N\setminus M_0$ has two connected components.
First, observe that a necessary condition for this is that $M_0$ is \emph{transversally oriented}
in $N$, i.e., the normal bundle $TM_0^\perp$ must be orientable. Thus, let us assume
that $M_0$ is transversally oriented.

The number of connected components is the rank of the free abelian group $\widetilde H_0(N\setminus M_0)$
plus $1$, where $\widetilde H_0$ denotes the reduced singular homology group.
The long exact reduced homology sequence of the pair $(N,N\setminus M_0)$ gives:
\begin{equation}\label{eq:longexactseq}
H_1(N)\longrightarrow H_1(N,N\setminus M_0)\longrightarrow\tilde H_0(N\setminus M_0)\longrightarrow
\tilde H_0(N)=0.
\end{equation}
Since $M_0$ is closed in $N$, we can use excision and replace $N$ with a tubular neighborhood of
$M_0$ in the term $H_1(N,N\setminus M_0)$; it follows that $H_1(N,N\setminus M_0)$ is the same as the
relative homology $H_1(TM_0^\perp,TM_0^\perp\setminus\mathbf0)$, where $TM_0^\perp$ is the normal bundle
of $M_0$ and $\mathbf 0$ is its zero section. The calculation of homology of vector bundles
is well known, see for instance \cite{Spanier}; since $M_0$ is transversally oriented, i.e.,
the vector bundle $TM_0^\perp$ is orientable, then $H_1(TM_0^\perp,TM_0^\perp\setminus\mathbf0)$
is isomorphic to $H_0(M_0)\cong\mathds Z$. Thus, we have an exact sequence:
\[H_1(N)\longrightarrow\mathds Z\longrightarrow\widetilde H_0(N\setminus M_0)\longrightarrow0.\]
A generator of the group $\mathds Z$ above is a curve in $N$ that intercepts once and transversally
$M_0$. Since $\widetilde H_0(N\setminus M_0)$ is free, the only options for the image of the
map $H_1(N)\to\mathds Z$ are that this image is either all $\mathds Z$ or zero. When this image
is $\mathds Z$, then $N\setminus M_0$ is connected; when the image is zero, then $N\setminus M_0$
has two connected components. Thus, we have the following:
\begin{lem}
Let $M$ and $N$ be compact connected manifolds, with $\mathrm{dim}(N)=\mathrm{dim}(M)+1$.
Then, the set:
\[\widetilde{\mathrm{Emb}}^o(M,N)\!=\!\Big\{[x]\in\widetilde{\mathrm{Emb}}(M,N):
N\setminus x(M)\ \text{has two connected components}\!\Big\}\]
is open in $\widetilde{\mathrm{Emb}}(M,N)$.
\end{lem}
\begin{proof}
The openness of ${\mathrm{Emb}}^o(M,N)$ in ${\mathrm{Emb}}(M,N)$ follows readily from the discussion above, observing that the orientability of the normal bundle
of an embedding $x:M\to N$ is stable by $C^1$-perturbations, while the homology class of $x$ is stable
by $C^0$-perturbations. Then, also $\widetilde{\mathrm{Emb}}^o(M,N)$ is open in $\widetilde{\mathrm{Emb}}(M,N)$,
because $\widetilde\pi$ is a quotient map and $\widetilde\pi^{-1}\big(\widetilde{\mathrm{Emb}}_1^o(M,N)\big)={\mathrm{Emb}}_1^o(M,N)$.
\end{proof}

\noindent
If $x:M\to N$ is a transversally oriented $C^1$-embedding, thus $[x]\in\widetilde{\mathrm{Emb}}^o(M,N)$,
then one can write $N\setminus x(M)=\Omega^1_x\bigcup\Omega^2_x$ as the disjoint union of two non empty
connected open subsets of $N$. Using an orientation of the normal bundle $x^\perp$ that depends
continuously by $C^1$-perturbations of $x$, one can define a continuous functions $\mathcal V^1$
and $\mathcal V^2$ in the connected
component of $x$ in $\widetilde{\mathrm{Emb}}_1^o(M,N)$, by setting:
\[\mathcal V^i(x)=\mathrm{volume}(\Omega^i_x)=\int_{\Omega^i_x}\mathrm{vol}_g,\qquad i=1,2.\]
Clearly, $\mathcal V^2(x)=\mathrm{vol}(N)-\mathcal V^1(x)$.
\subsection{A generalized volume functional}
Now, we observe that can write an alternative expression for the volume function
without an explicit reference to the set $\Omega^1_x$. Let us now assume that $M$ and $N$  are
oriented; then, one has a canonical choice of a transverse orientation of $x(M)$.
Let $U$ be an open subset $N$ that contains $x(M)$, and such that\footnote{%
Note that, by Sard theorem, since $\mathrm{dim}(M)<\mathrm{dim}(N)$ and $x$ is of class $C^1$, then
$x(M)\ne N$. Thus, one can take $U=N\setminus\{p\}$, where $p$ is any point of $N$ that does not
belong to $x(M)$.}  $U\ne N$; for instance, $U$ can
be taken to be a small tubular neighborhood of $x(M)$. Then, the volume form $\mathrm{vol}_g$, which is closed,
must be exact when restricted to $U$, because non compact manifolds have vanishing de Rham cohomology
in the highest dimension. This means that there exists an $(n-1)$-form $\eta_U$ on $U$ such that
$\mathrm d\eta_U=\mathrm{vol}_g$ on $U$. Then, one can define:
\[\mathcal V(x)=\int_Mx^*(\eta_U);\]
such a function is immediately seen to be independent of the choice of the primitive $\eta_U$.
By Stokes theorem, if $x(M)$ is the boundary of an open subset $\Omega$ of $N$, then
$\mathcal V(x)$ coincides with the volume of $\Omega$ or of $N\setminus\Omega$ (depending on the orientation of
$M$), and
thus the function $\mathcal V$ is a natural extension of the functions $\mathcal V^1$ and
$\mathcal V^2$ when no assumption is made on the number of connected components of $N\setminus x(M)$.
Note however that, when $N\setminus x(M)$ has only one connected component, the value of the
functional $\mathcal V$ does indeed depend on the choice of the open subset $U$ of $N$ that contains
$x(M)$. It can be observed also that $\mathcal V$ is invariant by right-composition with diffeomorphisms
of $M$, so that it gives a well defined function on a neighborhood of $[x]$ in $\widetilde{\mathrm{Emb}}(M,N)$.
We need a reformulation of some standard results of CMC embeddings (see \cite{BardoCEsc}) in our
abstract variational context:
\begin{prop}\label{thm:varproblCMC}
Let $x:M\to N$ be a transversally oriented smooth embedding, and let $U$ and $\eta$ be as above; consider the local
chart $\widetilde\Phi_x$ of $\widetilde{\mathrm{Emb}}(M,N)$ having domain $\widetilde{\mathcal U}_{x}$.
We will assume here that $k\ge2$.
\begin{itemize}
\item[(a)] The (locally defined) function $\mathcal V_x=\mathcal V\circ\widetilde\Phi_x^{-1}$
on $\widetilde\Phi_x\Big(\widetilde{\mathcal U}_{x}\Big)$ is smooth, and it has no critical
points.
\item[(b)] The critical points of the function $\mathcal A_x=\mathcal A\circ\widetilde\Phi_x^{-1}$
subject to the constraint $\mathcal V_x=\text{const.}$ are smooth sections $u$ of $x^\perp$
such that $\widetilde\Phi_x(u)$ is the class $[y]$ of a smooth embedding $y:M\to N$ having
constant mean curvature.
\item[(c)] If $[x]$ is a critical point of $\mathcal A_x=\mathcal A\circ\widetilde\Phi_x^{-1}$
subject to the constraint $\mathcal V_x=\text{const.}$, then its Lagrange multiplier $\lambda_x$ is equal
to $m\cdot H_x$, where $H_x$ is the (constant) mean curvature of $x$.
\item[(d)] If $[x]$ is a critical point of $\mathcal A_x$ subject to the constraint $\mathcal V_x=\text{const.}$, then identifying\footnote{%
Let $\vec n_x$ be positively orient unit normal field of the embedding $x$. An identification
of $\mathbf\Gamma(x^\perp)$ and $C^{k,\alpha}(M,\R)$ is given by $C^{k,\alpha}(M,\R)\ni f\mapsto f\cdot\vec n_x\in\mathbf\Gamma(x^\perp)$.}
$\mathbf\Gamma(x^\perp)$
with $C^{k,\alpha}(M,\R)$, the space of $C^{k,\alpha}$-functions on $M$,
the second variation $\mathrm d^2\mathcal A_x(x)$
at the point $x$ is the symmetric bilinear form on $C^{k,\alpha}(M,\R)$ corresponding to the
quadratic form:
\[Q_x(f)=\int_M-f\,\Delta f-\big(m\,\mathrm{Ric}_N(\vec n_x)+\Vert S_x\Vert^2\big)\,f^2\;\mathrm{vol}_x,\]
where $\Delta$ is the Laplacian on functions on $M$ relative to the induced metric $x^*(g)$,
$\mathrm{Ric}_N(\vec n_x)$ is the Ricci curvature of $N$ in the direction $\vec n_x$, which is the
positively oriented unit normal field of $x$, $S_x$ is the second fundamental form of $x$ in the direction $\vec n_x$,
$\Vert\cdot\Vert$ is the Hilbert--Schmidt norm of an operator defined by $g$,
and $\mathrm{vol}_x$ is the volume form of the
metric $x^*(g)$.
\end{itemize}
\end{prop}
\begin{proof}
The smoothness of $\mathcal V$ follows by the same argument used in the proof of the smoothness
of $\mathcal A$ in Proposition~\ref{thm:critpointsA}: $\mathcal V$ is smooth
map on $\mathrm{Emb}(M,N)$ invariant by diffeomorphisms of $M$, as it is in local charts the composition of a nonlinear first order differential operator
in the space of $C^{k,\alpha}$-sections of $x^*(TN)$ and the linear map given by integration on $M$.
All the remaining statements are obtained directly from \cite{BardoCEsc}, observing that the functions
 $\mathcal A_x$ and $\mathcal V_x$ correspond to the functions $A$ and $V$ in \cite{BardoCEsc}.\footnote{%
 In the language of \cite{BardoCEsc}, the volume function denoted by $V$ is associated to \emph{variations} of
 $x$, i.e., to smooth maps $X:\left]-\varepsilon,\varepsilon\right[\times M\to N$, such that for all $t\in\left]-\varepsilon,\varepsilon\right[$,
 the map $X_t:=X(t,\cdot):M\to N$ is an immersion, and such that $X_0=x$. Note that, since $x$ is an embedding,
 for $t$ near $0$ also $X_t$ is an embedding. Associated to such a variation of $x$, it is defined in
 \cite{BardoCEsc} the function $V(t)=\int_{[0,t]\times M}X^*(\mathrm{vol}_g)$, for
 $t\in\left]-\varepsilon,\varepsilon\right[$. Now, using Stoke's Theorem:
 \[V(t)=\int_{[0,t]\times M}X^*(\mathrm d\eta)=\int_{[0,t]\times M}\mathrm dX^*(\eta)=
 \int_M(X_t)^*(\eta)-\int_M(X_0)^*(\eta)=\mathcal V(X_t)-\mathcal V(X_0).\]}
 The differential of $\mathcal V_x$ at $x$ is identified with the
 linear map $C^{k,\alpha}(M,\R)\ni f\mapsto\int_Mf\;\mathrm{vol}_x$, which is not identically zero;
 thus, $x$ is not a critical point of
 $\mathcal V_x$, and  it follows that   $\mathcal V_x$
 has no critical points, proving (a). Parts (b) and (c) follow by similar arguments from \cite[Proposition~2.3]{BardoCEsc}
 or \cite[Proposition~2.7]{BardoC}. Part (d) is \cite[Proposition~2.5]{BardoCEsc}.
\end{proof}
Let us now look at the Fredholmness issues for the CMC variational problem;
in order to comply with axiom (HF-B), \emph{we will henceforth assume that $k=2$}, i.e.,
we will consider the manifold of unparameterized embeddings of class $C^{2,\alpha}$,
with $\alpha\in\left]0,1\right[$.
Given a constant mean curvature embedding $x:M\to N$, let $J:C^{2,\alpha}(x^\perp)\to C^{0,\alpha}(x^\perp)$
be the differential operator:
\begin{equation}\label{eq:Jacobiperator}
Jf=-\Delta f-\big(m\,\mathrm{Ric}_N(\vec n_x)+\Vert S_x\Vert^2\big)\,f.
\end{equation}
From part (d) of Proposition~\ref{thm:varproblCMC}, the second variation of the functional
$\mathcal A_{x}+\lambda_x\mathcal V_{x}$ is the symmetric bilinear form $B_x$ on
$C^{2,\alpha}(x^\perp)$ given by:
\begin{equation}\label{eq:secvarsymmform}
B_x(f_1,f_2)=\int_M (Jf_1)\cdot f_2\;\mathrm{vol}_x.
\end{equation}
\begin{prop}\label{thm:FredhomnessCMC}
The following statements hold:
\begin{itemize}
\item[(a)] the functional $\mathcal A_x+\lambda_x\mathcal V_x$ satisfies the assumption
(HF-A) in a
neighborhood of $[x]$ in $\widetilde{\mathrm{Emb}}(M,N)$;
\item[(b)] the functional $\mathcal A_x+\lambda_x\mathcal V_x$ satisfies the assumption (HF-B) in a
neighborhood of $[x]$ in $\widetilde{\mathrm{Emb}}(M,N)$.
\end{itemize}
\end{prop}
\begin{proof}
Part (a) is well known, see for instance \cite[Section~2]{Smale},
\cite[\S 1.4]{Whi}, \cite[Theorem~1.2 and \S~7]{Whi2}.
The gradient map $H$ is defined in an open set of the Banach space $X=C^{2,\alpha}(M,\R)$
and takes values in the Banach space $Y=C^{0,\alpha}(M,\R)$.
The Hilbert space $H$ is the Lebesgue space of square integrable functions on $M$
relatively to the metric defined by $\mathrm{vol}$.
An explicit formula
for the map $H$ is irrelevant here; it is given by a quasi-linear second order elliptic differential
operator, see \cite{Whi, Whi2} for details.

For part (b), consider the embedding
$C^{2,\alpha}(M,\R)\hookrightarrow H^1(M,\R)$, where $H^1$ denotes the Sobolev space given by
the Hilbert space completion of $C^{2,\alpha}(M,\R)$ with respect to the pre-Hilbert space inner product
\[\langle f_1,f_2\rangle_{H^1}=\int_M\big[f_1f_2+\nabla f_1\cdot\nabla f_2\big]\;\mathrm{vol}_x\] (gradients and inner products are relative
to the Riemannian metric $x^*(g)$ on $M$).
The symmetric bilinear form $B_x=\mathrm d^2\big[\mathcal A_{x}+\lambda_x\mathcal V_{x}\big](x)$
has a bounded extension to a symmetric bilinear form on $H^1(M,\R)$, which is represented by an essentially positive
self-adjoint operator on $H^1(M,\R)$. Namely, partial integration in \eqref{eq:secvarsymmform} gives;
\begin{multline}\label{eq:Bx}
B_x(f_1,f_2)=\int_m\nabla f_1\cdot\nabla f_2-\big(m\,\mathrm{Ric}_N(\vec n_x)
+\Vert S_x\Vert^2\big)\,f_1\,f_2\;\mathrm{vol}_x\\=\langle f_1,f_2\rangle_{H^1}-
\int_M\big(m\,\mathrm{Ric}_N(\vec n_x)
+\Vert S_x\Vert^2+1\big)\,f_1\,f_2\;\mathrm{vol}_x,\end{multline}
from which it follows that $B_x$ has a bounded extension to $H^1$.
Now, such extension is represented by the operator $\mathrm I+K$, where $\mathrm I$ is the identity,
and:
\[\langle Kf_1,f_2\rangle_{H^1}= -\int_M\big(m\,\mathrm{Ric}_N(\vec n_x)+\Vert S_x\Vert^2+1\big)
\,f_1\,f_2\;\mathrm{vol}_x.\]
The bilinear form on the right-hand side of this equality is continuous with respect
to the $L^2$-topology, and this implies that $K$ is a compact operator, because the
inclusion $H^1\hookrightarrow L^2$ is compact. Hence, $S=\mathrm I+K$ is essentially positive.

A function $f\in H^1(M,\R)$ is an eigenfunction of $S$ if and only if it is a weak solution
of the linear elliptic equation $Jf=\mu f$ for some $\mu\in\R$. By standard elliptic regularity
(see for instance \cite[Chapter 8]{GilTrud}), every such $f$ is of class $C^\infty$, and
thus all the eigenspaces of $S$ are contained in $T_{[x]}\widetilde{\mathrm{Emb}}(M,N)$.
This concludes the proof.
\end{proof}
\end{section}
\begin{section}{Constant mean curvature Clifford tori in the sphere}
\label{sec:Clifford}
We will now finalize the proof of our main theorem, by applying the abstract results
discussed in the first part of the paper. Let $0<j<m$ be fixed integers,
and let us look at the case of CMC embeddings of the product
of sphere $M=\mathbb S^j\times\mathbb S^{m-j}$ into the sphere $N=\mathbb S^{m+1}$
endowed with the round metric $g$.

The Banach manifold $\mathfrak M$ is (an open subset of) $\widetilde{\mathrm{Emb}}(M,N)$,
here $\alpha$ is any real number in $\left]0,1\right[$. Consider the group $G=\mathrm{SO}(m+2)$, which
is the connected component of the identity of $\mathrm O(m+2)=\mathrm{Iso}(N,g)$,
acting by left-composition on $\mathfrak M$. For a given $x\in\widetilde{\mathrm{Emb}}(M,N)$,
the stabilizer $G_x$ is given by the set of isometries $\psi\in\mathrm{SO}(m+1)$ that preserve
the subset $x(M)$, i.e., such that $\psi\big(x(M)\big)=x(M)$.

Consider the path $\left]0,1\right[\ni r\mapsto\big[x_r^{m,j}\big]\in\widetilde{\mathrm{Emb}}(M,N)$
of CMC Clifford tori defined in \eqref{eq:defClifford}. If $g$ denotes the round metric of radius $1$
on the sphere $\mathbb S^{m+1}$, the pull-back $\big(x_r^{m,j}\big)^*(g)$ is the product metric
on $\mathbb S^j\times\mathbb S^{m-j}$ which is the round metric of radius $r$ on the factor $\mathbb S^j$
and the round metric of radius $\sqrt{1-r^2}$ on the factor $\mathbb S^{m-j}$. We will
denote this Riemannian manifold by $\mathbb S^j(r)\times\mathbb S^{m-j}\big(\sqrt{1-r^2}\big)$.
\begin{prop}\label{thm:conststabilizer}
The connected component of the identity of the subgroup of $\mathrm{SO}(m+2)$ that stabilizes
$\big[x_r^{m,j}\big]$ in $\widetilde{\mathrm{Emb}}(M,N)$ is $\mathrm{SO}(j+1)\times\mathrm{SO}(m-j+1)$
(embedded diagonally in $\mathrm{SO}(m+2)$).
\end{prop}
\begin{proof}
Such a connected component
is compact, and it obviously contains the product $\mathrm{SO}(j+1)\times\mathrm{SO}(m-j+1)$. On the other hand,
it is not equal to $\mathrm{SO}(m+2)$ if $j$ and $m-j$ are positive. But
$H=\mathrm{SO}(j+1)\times\mathrm{SO}(m-j+1)$ is a \emph{maximal} connected subgroup of $\mathrm{SO}(m+2)$
(see for instance \cite{Dyn}), and thus it must be equal to the connected component of the identity of the
stabilizer of $\big[x_r^{m,j}\big]$.
\end{proof}
\begin{cor}\label{thm:dimensionorbit}
The $\mathrm{SO}(m+2)$-orbit of the class $\big[x_r^{m,j}\big]$ in
$\widetilde{\mathrm{Emb}}(\mathbb S^j\times\mathbb S^{m-j},\mathbb S^m)$ is diffeomorphic
to the Grassmannian of all $(j+1)$-dimensional oriented subspaces of $\mathds R^{m+2}$, whose
dimension is equal to $m+1+j(m-j)$.
\end{cor}
\begin{proof}
The orbit of $\big[x_r^{m,j}\big]$ is diffeomorphic to the quotient
\[\mathrm{SO}(m+2)/\big[\mathrm{SO}(j+1)\times\mathrm{SO}(m-j+1)\big],\] whose dimension is
$\frac12(m+2)(m+1)-\frac12j(j+1)-\frac12(m-j)(m-j+1)=m+1+j(m-j)$.
\end{proof}
The path of CMC Clifford tori is never tangent to the orbits:
\begin{prop}\label{thm:pathtransverse}
The curve $r\mapsto\big[x_r^{m,j}\big]$ is not tangent to
the orbit $\mathrm{SO}(m+2)\big[x_{r_0}^{m,j}\big]$ for any value of $r_0$.
\end{prop}
\begin{proof}
Fix $r_0$ and denote by $P^\perp:\mathbf\Gamma\big((x_{r_0}^{m,j})^*(T\mathbb S^{m+1})\big)\to
\mathbf\Gamma\big((x_{r_0}^{m,j})^\perp\big)$
the linear map that carries a vector field $V$ along $x^{m,j}_{r_0}$ to the normal field $V^\perp$
obtained by pointwise orthogonal projection onto the normal space of $x^{m,j}_{r_0}$.
The tangent space to the orbit of $\big[x_r^{m,j}\big]$ is
identified (via the local chart $\widetilde\Phi_{x_{r_0}^{m,j}}$) with the space:
\[\Big\{P^\perp\big(\mathfrak h\circ x_{r_0}^{m,j}\big):\mathfrak h\in\mathfrak{so}(m+2)\Big\},\]
where $\mathfrak{so}(m+2)$ is the Lie algebra of $\mathrm{SO}(m+2)$ (consisting of anti-symmetric
matrices), while the tangent vector to the curve of CMC Clifford tori is:
\[P^\perp\left(\frac{\mathrm d}{\mathrm dr}\Big\vert_{r=r_0}x_r^{m,j}\right),\]
where clearly:
\[\frac{\mathrm d}{\mathrm dr}\Big\vert_{r=r_0}x_r^{m,j}(p,q)=\left(p,-\frac{r_0}{\sqrt{1-r_0^2}}\,q\right),\qquad
\forall\,p\in\mathbb S^j,\ q\in\mathbb S^{m-j}.\]
The thesis is equivalent to proving that there exists no $\mathfrak h\in\mathfrak{so}(m+2)$ such
that the vector field:
\begin{equation}\label{eq:condtangent}
\mathfrak h\circ x_{r_0}^{m,j}-\frac{\mathrm d}{\mathrm dr}\Big\vert_{r=r_0}x_r^{m,j}
\end{equation}
is everywhere tangent to the Clifford torus $x^{m,j}_{r_0}$. Writing:
\[\mathfrak h=\begin{pmatrix}A&B\\ -B^t&C\end{pmatrix},\]
where $A$ is an anti-symmetric $(j+1)\times(j+1)$ matrix, $B$ is a $(j+1)\times(m-j+1)$ matrix and
$C$ is an anti-symmetric $(m-j+1)\times(m-j+1)$ matrix, then \eqref{eq:condtangent} is everywhere
tangent to the Clifford torus if and only if:
\[Bq\cdot p=\frac1{\sqrt{1-r_0^2}},\quad\forall\,p\in\mathbb S^j,\ q\in\mathbb S^{m-j};\]
clearly, such condition is not satisfied by any $B$, which proves the thesis.
\end{proof}
Let us now study the nondegeneracy and the Morse index of the CMC Clifford tori.
Consider the sequences $\big(\beta_i(j)\big)_{i\ge3}$ and $\big(\gamma_l(j,m)\big)_{l\ge3}$
defined by:
\begin{equation}\label{eq:betagamma}
\beta_i(j)=(i-2)(j+i-1),\qquad \gamma_l(j,m)=(l-2)(m-j+l-1).
\end{equation}
It is easy to see that they are strictly increasing, and that:
\[\lim_{i\to\infty}\beta_i(j)=\lim_{l\to\infty}\gamma_l(j,m)=+\infty.\]
Define two sequences $\big(r_i^{m,j}\big)_{i\ge3}$ and $\big(s_l^{m,j}\big)_{l\ge3}$ by:
\begin{equation}\label{eq:jumpsofindex}
r_i^{m,j}=\sqrt{\frac{\beta_i(j)}{m-j+\beta_i(j)}}, \qquad s_l^{m,j}=\sqrt{\frac j{j+\gamma_l(j,m)}};
\end{equation}
the sequence $r_i^{m,j}$ is contained in $\left[\sqrt{\frac{j+2}{m+2}},1\right[$, it is strictly increasing, and $\lim\limits_{i\to\infty}r_i^{m,j}=1$, while $s_l^{m,j}$ is contained in $\left]0,\sqrt{\frac j{m+2}}\;\right]$, it is
strictly decreasing, and $\lim\limits_{l\to\infty}s_l^{m,j}=0$.
\begin{prop}\label{thm:deg-Morseindex}
The following statements hold:
\begin{itemize}
\item[(a)] the strong Morse index of $\big[x_r^{m,j}\big]$ is equal to the weak Morse index plus $1$;
\smallskip

\item[(b)] the $\mathrm{SO}(m+1)$-critical orbits of $\big[x_r^{m,j}\big]$ is degenerate
if and only if either $r=r_i^{m,j}$ for some $i\ge3$ or $r=s_l^{m,j}$ for some $l\ge3$;
\smallskip

\item[(c)] each degeneracy instant
of the path of critical points $\left]0,1\right[\ni r\mapsto\big[x_r^{m,j}\big]$
determines a jump of the Morse index.

\end{itemize}
More precisely, at every degeneracy instant $r_i^{m,j}$, the jump of the Morse index is in absolute value equal to:
\[\begin{pmatrix}j+i-1\cr i-1\end{pmatrix}-\begin{pmatrix}j+1-3\cr i-3\end{pmatrix},\]
while at every degeneracy instant $s_l^{m,j}$, the jump of the Morse index is in absolute value equal to:
\[\begin{pmatrix}m-j+l-1\cr l-1\end{pmatrix}-\begin{pmatrix}m-j+l-3\cr l-3\end{pmatrix}.\]
\end{prop}
\begin{proof}
The proof is based on a direct analysis of the spectrum of the Jacobi operator $J$ in \eqref{eq:Jacobiperator}, following \cite{LuisAldirOscar}. The Ricci curvature of the sphere $\mathbb S^{m+1}$ is constant equal to $1$; also the
norm of the second fundamental form of the CMC Clifford embedding $x_r^{m,j}$ is constant, and the following formula holds:
\[m\,\mathrm{Ric}_N(\vec n_{x_r^{m,j}})+\big\Vert S_{x_r^{m,j}}\big\Vert^2\equiv\frac j{r^2}+\frac{m-j}{1-r^2}.\]
Thus, the Jacobi operator takes the form:
\[J=-\Delta^{j,m-j}_r-\left(\frac j{r^2}+\frac{m-j}{1-r^2}\right),\]
where $\Delta^{j,m-j}_r$ is the Laplacian of the Riemannian manifold
$\mathbb S^j(r)\times\mathbb S^{m-j}\big(\sqrt{1-r^2}\big)$. Thus, the spectrum of
$J$ is given by the spectrum of $-\Delta^{j,m-j}$ shifted by $-\left(\frac j{r^2}+\frac{m-j}{1-r^2}\right)$;
moreover, $J$ and $\Delta^{j,m-j}$ have the same eigenfunctions.
Degeneracy and strong Morse index of the critical point $\big[x_r^{m,j}\big]$ is studied
by counting the number of zero and negative eigenvalues of $J$ in the space of (smooth) real functions
on $\mathbb S^j\times\mathbb S^{m-j}$. The weak Morse index is given by the index of the
restriction of the quadratic form $\langle Jf,f\rangle_{L^2}$ to the space of functions $f$ on
$\mathbb S^j\times\mathbb S^{m-j}$ having vanishing integral.

An eigenvalue of $J$ has the form:
\begin{equation}\label{eq:eigenJ}
\sigma_i+\rho_l-\left(\frac j{r^2}+\frac{m-j}{1-r^2}\right),
\end{equation}
where $\sigma_i$ is an eigenvalue of the Laplacian $\Delta^j_r$ of the sphere $\mathbb S^j(r)$
and $\rho_l$ is an eigenvalue of the Laplacian $\Delta^{m-j}_{\sqrt{1-r^2}}$ of the sphere $\mathbb S^{m-j}\big(\sqrt{1-r^2}\big)$.
Moreover, if $M_{\sigma}$ and $M_{\rho}$ are the multiplicity of the eigenvalues $\sigma$ and $\rho$ respectively
of $\Delta^j_r$ and of $\Delta^{m-j}_{\sqrt{1-r^2}}$, then the multiplicity of \eqref{eq:eigenJ} is given by
$\sum M_\sigma M_\rho$, where the sum is taken over all eigenvalues $\sigma$ and $\rho$ such that
$\sigma+\rho=\sigma_i+\rho_l$.

The $\sigma_i$'s and the $\rho_l$'s form two strictly increasing unbounded sequences, and they have
multiplicities denoted respectively by $M_{\sigma_i}$ and $M_{\rho_l}$ given by the following formulas:
\begin{equation}\label{eq:sigmarho}
\sigma_i=\frac{(i-1)(j+i-2)}{r^2},\qquad\rho_l=\frac{(l-1)(m-j+l-2)}{1-r^2}
\end{equation}
\begin{align}\label{eq:multiplicities}
&&M_{\sigma_1}=1,\quad M_{\sigma_2}=j+1,\quad M_{\sigma_i}=\begin{pmatrix}j+i-1\cr i-1\end{pmatrix}
-\begin{pmatrix}j+i-3\cr i-3\end{pmatrix}\ \text{for $i\ge3$},
\\ \notag
&&
M_{\rho_1}=1,\quad M_{\rho_2}=m-j+1,\quad M_{\rho_l}=\begin{pmatrix}m-j+l-1\cr l-1\end{pmatrix}
-\begin{pmatrix}m-j+l-3\cr l-3\end{pmatrix}\\ \notag &&\text{for $l\ge3$}.
\end{align}
The negative eigenvalue $-\left(\frac j{r^2}+\frac{m-j}{1-r^2}\right)=\sigma_1+\rho_1-\left(\frac j{r^2}+\frac{m-j}{1-r^2}\right)$
has multiplicity $1$, and its eigenspace consists of constant functions. This gives a contribution of $1$ to
the strong Morse index of $\big[x_r^{m,j}\big]$, but not to the weak Morse index.
On the other hand, all the other eigenspaces of $J$ are $L^2$-orthogonal to the first eigenspace constisting
of constant functions on $\mathbb S^j\times\mathbb S^{m-j}$, i.e., non constant eigenfunctions of
$J$ have vanishing integral. This implies that the strong Morse index $\big[x_r^{m,j}\big]$
is equal to the weak Morse index plus $1$, proving (a).

As to the degeneracy, let us observe that $0$ is always an eigenvalue of $J$, because:
\[\sigma_2+\rho_2-\left(\frac j{r^2}+\frac{m-j}{1-r^2}\right)=0,\]
for all $r\in\left]0,1\right[$.
Assuming that there is no other pair $(i,l)\ne(2,2)$ such that $\sigma_i+\rho_l-
\left(\frac j{r^2}+\frac{m-j}{1-r^2}\right)=0$, then the multiplicity of zero as an eigenvalue of $J$ is given by:
\[M_{\sigma_2}M_{\rho_2}=(j+1)(m-j+1)=m+1+j(m-j).\]
By Corollary~\ref{thm:dimensionorbit}, this multiplicity equals the dimension of the critical orbit
of $\big[x_r^{m,j}\big]$. Thus, nondegeneracy occurs exactly when there is no other $0$ in the
spectrum of $J$. Since the sequences $\sigma_i$ and $\rho_l$ are strictly increasing, it follows that
other $0$'s in the spectrum of $J$ must be of the form \[\sigma_1+\rho_l-\left(\frac j{r^2}+\frac{m-j}{1-r^2}\right)\]
for some $l\ge3$, or of the form \[\sigma_i+\rho_1-\left(\frac j{r^2}+\frac{m-j}{1-r^2}\right)\] for some
$i\ge3$. Recalling \eqref{eq:betagamma} and \eqref{eq:sigmarho}, we compute explicitly:
\begin{multline*}
\sigma_1+\rho_l-\left(\frac j{r^2}+\frac{m-j}{1-r^2}\right)=
\frac{(l-1)(m-j+l-2)}{1-r^2}-\left(\frac j{r^2}+\frac{m-j}{1-r^2}\right)\\
=\frac{r^2\big[j+\gamma_l(j,m)\big]-j}{r^2(1-r^2)},
\end{multline*}
\begin{multline*}
\sigma_i+\rho_1-\left(\frac j{r^2}+\frac{m-j}{1-r^2}\right)=
\frac{(i-1)(j+i-2)}{r^2}-\left(\frac j{r^2}+\frac{m-j}{1-r^2}\right)\\
=\frac{\beta_i(j)-r^2\big[m-j+\beta_i(j)\big]}{r^2(1-r^2)}.
\end{multline*}
The first expression above vanishes for $r\in\left]0,1\right[$ exactly when $r=s_l^{m,j}$, while the second expression vanishes for $r\in\left]0,1\right[$ exactly when $r=r_i^{m,j}$, see \eqref{eq:jumpsofindex}. This proves part (b).

For the proof of part (c), we will show that the jump of the Morse index at the instants $s_l^{m,j}$ is equal to the multiplicity $M_{\rho_l}$,
while the jump of the Morse index at the instants $r_i^{m,j}$ is equal to $M_{\sigma_i}$,
see \eqref{eq:multiplicities}. This amounts to proving that the zeros of the functions:
\[\theta_l(r)=\frac{r^2\big[j+\gamma_l(j,m)\big]-j}{r^2(1-r^2)}\]
and
\[\kappa_i(r)=\frac{\beta_i(j)-r^2\big[m-j+\beta_i(j)\big]}{r^2(1-r^2)}\]
in the interval $\left]0,1\right[$ are points where the functions change their
sign. An elementary analysis of these functions show that $\theta_l$ is strictly increasing
and $\kappa_i$ is strictly decreasing on $\left]0,1\right[$ for all $l,i\ge3$.
Thus, every degeneracy instant
determines a jump of the Morse index equal to its multiplicity.
\end{proof}
The constant mean curvature torus $x^{m,j}_r$ has mean curvature $H^{m,j}_r$ given by (see \cite{LuisAldirOscar}):
\[H^{m,j}_r=\frac{mr^2-j}{mr\sqrt{1-r^2}};\]
thus, by Proposition~\ref{thm:varproblCMC}, part (c), the corresponding Lagrange multiplier $\lambda^{m,j}_r$ is:
\[\lambda_r^{m,j}=\frac{mr^2-j}{r\sqrt{1-r^2}}.\]
The derivative $\frac{\mathrm d}{\mathrm dr}\lambda_r^{m,j}$ is positive:
\begin{equation}\label{eq:positivederivative}
\frac{\mathrm d}{\mathrm dr}\lambda_r^{m,j}=\frac{(m-2j)r^2+j}{r^2(1-r^2)^{3/2}}>0,\end{equation}
because $(m-2j)r^2+j\ge\min\{j,m-j\}>0$ for $r\in\left]0,1\right[$.

\begin{rem}\label{rm:minimalClifford}
Note that the value of $r=\sqrt{\frac jm}\in\left]\sqrt{\frac j{m+2}},\sqrt{\frac{j+2}{m+2}}\right[$,
corresponding to the \emph{minimal} Clifford torus, is not a degeneracy instant for
the family of CMC Clifford tori.
\end{rem}

Finally, everything is now ready for:
\begin{proof}[Proof of Theorem]
The existence of CMC embeddings of the product $\mathbb S^j\times\mathbb S^{m-j}$ into
$\mathbb S^{m+1}$ that accumulate at the CMC Clifford tori $x^{m,j}_{r_i}$ and $x^{m,j}_{s_i}$
and that are not congruent to any member of the CMC Clifford family corresponds exactly to
the occurrence of orbit bifurcation at $r=r_i$ and at $r=s_i$ for the CMC constrained variational
problem on the manifold of unparameterized embeddings of $\mathbb S^j\times\mathbb S^{m-j}$ into
$\mathbb S^{m+1}$.
The existence result is obtained by applying Theorem~\ref{thm:maincriticalorbitbifurcation}, whose
assumptions are satisfied, as follows.
The topological and differential structure for the manifold $\mathfrak M=\widetilde{\mathrm{Emb}}(M,N)$
the area and the volume functional $\mathcal A$ and $\mathcal V$
required in assumptions~(A1)---(A5) for the CMC variational problem
has been discussed in Section \ref{sec:varproblem}. The path
of classes of CMC Clifford tori $r\mapsto\big[x^{m,j}_r\big]$ satisfies axioms (B1) and (B2);
see \eqref{eq:positivederivative} for the condition $\lambda_r'>0$.

The statements in assumption (HF-A) and (HF-B) are established respectively in part (a) and (b) of
Proposition~\ref{thm:FredhomnessCMC}. The continuity assumption (HF-B3) is deduced immediately from
\eqref{eq:Bx}.

The statement of assumption (C1) is proved in Proposition~\ref{thm:conststabilizer}.
As to assumption (C2), this follows from Proposition~\ref{thm:pathtransverse}.

The hypotheses (D1) and (D2) hold by Proposition~\ref{thm:deg-Morseindex}.

Hence, by Theorem~\ref{thm:maincriticalorbitbifurcation} there is orbit bifurcation at the instants
$r=r_i$ and $s=s_i$, $i\in\mathds N$.

As to the local rigidity of the CMC Clifford family for all other values of $r$, this follows readily
as an application of Proposition~\ref{thm:bifimpliesdegen} to the above setup, using the fact
that, by Proposition~\ref{thm:deg-Morseindex}, for all values of $r$ that do not belong to the sequences
$r_i$ and $s_i$, the critical orbit $\mathrm{SO}(m+2)\big[x^{m,j}_r\big]$ is nondegenerate.
This concludes the proof.
\end{proof}

\end{section}

\end{document}